\documentclass[journal]{IEEEtran}
\usepackage[draft=false, colorlinks=true, allcolors = blue]{hyperref}
\usepackage[caption=false,font=footnotesize]{subfig}
\usepackage{amsmath,amsfonts}
\usepackage[ruled,vlined]{algorithm2e}
\usepackage{array}
\usepackage{textcomp}
\usepackage{url}
\usepackage{verbatim}
\usepackage{graphicx}
\usepackage{cite}
\usepackage{amssymb}

\usepackage{graphicx}

\usepackage{color}
\usepackage{bm}

\usepackage{amsthm}
\usepackage{amsmath}
\newtheorem{proposition}{Proposition}[section]
\newtheorem{lemma}{Lemma}[section]
\newtheorem{theorem}{Theorem}[section]
\newtheorem{fact}{Fact}[section]

\theoremstyle{definition}
\newtheorem{definition}{Definition}[section]
\newtheorem{remark}{Remark}[section]
\newtheorem{assumption}{Assumption}[section]
\newtheorem{example}{Example}[section]

\newcommand{\diagvec}[1]{{\mathrm{diag}(\boldsymbol{#1})}}
\newcommand{\signal}[1]{{\boldsymbol{#1}}}

\newcommand{\real}{{\mathbb R}}

\newcommand{\Natural}{{\mathbb N}}
\newcommand{\refeq}[1]{(\ref{#1})}

\newcommand{\fnorm}{{F_{\|\cdot\|}}}

\newcommand{\spradG}{{\varrho_{G}}}
\theoremstyle{remark}
\usepackage{setspace}
\usepackage{hyperref}
\DeclareMathOperator*{\minimize}{minimize}

\DeclareMathOperator*{\argmin}{argmin}
\DeclareMathOperator*{\argmax}{argmax}

\begin{document}
	
	\title{Sum-Rate Maximization via Convex Optimization Using Subgradient Projections Onto Nonlinear Spectral Radius Constraint Sets}
	
	\author{Hiroki Kuroda,~\IEEEmembership{Member,~IEEE,}~and~Renato L.~G.~Cavalcante,~\IEEEmembership{Member,~IEEE}%
	\thanks{The work of the first author was supported in part by the Japan Society for the Promotion of Science under Grants-in-Aid 21K17827. The second author acknowledges the 6G-MIRAI-Harmony project, which has received funding from the Smart Networks and Services Joint Undertaking (SNS JU) under the European Union's Horizon Europe research and innovation programme, Grant Agreement No. 10119236. Views and opinions expressed are, however, those of the authors only and do not necessarily reflect those of the European Union or the SNS JU (granting authority). Neither the European Union nor the granting authority can be held responsible for them. {\it (Corresponding author: Hiroki Kuroda.)}}%
	\thanks{Hiroki Kuroda is with Nagaoka University of Technology, Niigata 940-2188, Japan (e-mail: kuroda.h@ieee.org). He conducted part of this work as a guest researcher at Fraunhofer Heinrich Hertz Institute.}%
	\thanks{Renato L.~G.~Cavalcante is with the Fraunhofer Heinrich Hertz Institute, 10587 Berlin, Germany (e-mail: renato.cavalcante@hhi.fraunhofer.de)}%
	}
	
	\maketitle

	\begin{abstract}
    We solve the (weighted) sum-rate maximization problem over the set of achievable rates characterized by a nonlinear spectral radius function. This set has been recently shown to be convex in some practically relevant settings in modern wireless networks, including cell-less networks. However, even under convexity, sum-rate maximization is challenging because the nonlinear spectral radius characterization of the achievable rate region is difficult to handle directly. We overcome this difficulty by exploiting subgradient projections onto the level sets of suitably reformulated spectral radius functions. Notably, the derived subgradient projection algorithm provably converges to the global optimum of the sum-rate maximization problem under the convexity condition. The efficacy of the proposed algorithm is illustrated in simulations for cell-less networks.
	\end{abstract}
	 

\section{Introduction}
\label{sect.intro}
Weighted sum-rate maximization has a long-standing history in the wireless communications literature, with applications spanning cross-layer control, link scheduling, and power control, to name a few (see \cite{weeraddana2012} for an overview). Early studies on sum-rate maximization \cite{luo2008dynamic,weeraddana2012,liu2012achieving,zheng14,tan2011nonnegative,friedland2008maximizing,chiang07,tan2013,shi2011iteratively} typically considered instantaneous rates defined by individual channel realizations. More recently, driven in large part by advances in the massive MIMO and cell-less literature, the focus has shifted toward maximizing tractable bounds on Shannon achievable (ergodic) rates, most notably the use-and-then-forget (UatF) bound \cite{marzetta16,bjornson2017massive,miretti2024sum,miretti2025two}. In these formulations, the optimization problems are solved once per channel distribution \cite{miretti2024sum,bjornson2017massive,marzetta16,miretti2025two}, rather than once per channel realization as in the earlier approaches. This reduced dependence on instantaneous channel state information is especially valuable in distributed deployments (e.g., cell-less networks), where a per-realization resource-allocation strategy would entail acquiring channel state information from many geographically separated nodes and repeatedly solving an optimization problem on the channel-coherence time scale -- an overhead often prohibitive in practice.

From an abstract mathematical viewpoint, formulations based on instantaneous rates and those based on trackable bounds on ergodic rates are closely connected. In either case, the resulting weighted sum-rate maximization problems are (strongly) NP-hard in general \cite{luo2008dynamic}, which forces a tradeoff between provably optimal algorithms that are typically slow and fast heuristics that offer no optimality guarantees \cite{weeraddana2012}. 
Representative methods in the former category include branch-and-bound schemes \cite{xu2008global} and the polyblock algorithm \cite{liu2012achieving}.  Unfortunately, the complexity of global approaches becomes prohibitive even for modest problem sizes. For instance, the polyblock algorithm is reported to be impractical beyond roughly six users \cite[Sect.~VI]{liu2012achieving}.
For this reason, the literature largely relies on fast heuristics, with proposed techniques spanning convex relaxations \cite[Sect.~IV]{zheng14},\cite[Sect.~4]{tan2011nonnegative}, \cite{friedland2008maximizing}, geometric programming formulations \cite{chiang07,tan2013}, and the weighted minimum mean-square error (WMMSE) method \cite{shi2011iteratively}, among others.

Despite the NP-hardness of sum-rate maximization in general, some special cases can be cast as convex optimization problems. Notable examples include the low transmit power regime \cite[Sect.~V]{tan2011maximizing},\cite[Sect.~III]{zheng14} and settings in which users experience equal interference-plus-noise \cite[Sect.~V.A]{cheng2017optimal}. More recently, the study in \cite{renatoconvex} has lifted some of these assumptions by focusing on interference models based on the UatF bound. By using the spectral radius of nonlinear mappings, that study identifies practical conditions in which the set of achievable rates becomes convex. However, even when convexity holds, designing efficient algorithms tailored to the resulting formulations have remained elusive, largely owing to the intricate structure of the achievable rate region, as detailed below.

\subsection{Potentially tractable classes of the sum-rate maximization problems and their associated challenges}
\label{sect.intro.reformulation}
Typically, weighted sum-rate maximization -- whether based on instantaneous signal-to-interference-plus-noise ratio (SINR) or the UatF bound -- can be written in the following unified form (see Sect.~\ref{sect.notation} for mathematical definitions):

\begin{align}
	\label{eq.sumrate.power}\tag{P}
		\minimize_{\signal{p}\in\mathcal{P}}~  {} -\sum_{n \in \mathcal{N}}w_n\log\left(1 + \frac{p_n}{f_n(\signal{p})}\right),
\end{align}
where $\mathcal{P}$ is the power constraint set given by
\begin{align}
	\label{eq.C}
	\mathcal{P} := \{\signal{p}\in\real_+^N\mid \|\signal{p}\|\le 1\}
\end{align}
for a given \emph{order-preserving} norm $\|\cdot\|$,
$\mathcal{N} := \{1,\ldots,N\}$ denotes the set of $N \in \Natural$ users,
$w_n \in \real_{++}$ is a positive weight (priority) assigned to user $n \in \mathcal{N}$, and
$f_n\colon \real_+^N\to\real_{++}$ is a \emph{strictly-subhomogeneous order-preserving continuous (SSOC)} function --
called \emph{standard interference function} in the wireless literature \cite{yates95,martin11} --
for every $n \in \mathcal{N}$.

Problem \eqref{eq.sumrate.power} can be
equivalently cast as
\begin{align}
	\label{eq.sumrate.sinr}\tag{S}
		\minimize_{\signal{s}\in\mathcal{S}}~  {} - \sum_{n \in \mathcal{N}}w_n\log(1 + s_n),
\end{align}
where the \emph{achievable SINR region} $\mathcal{S}$, i.e., the set
of achievable SINRs of users, is given by
\begin{align*}
\nonumber
\mathcal{S} := {} & \left\{\left. \left(\frac{p_n}{f_n(\signal{p})}\right)_{n \in \mathcal{N}} \in \mathbb{R}_{+}^{N} 
\,\right|\, \signal{p} \in \mathcal{P} \right\}.
\end{align*}
Under a fairly general assumption (see Sect.~\ref{sect.general.sumrateprob} for details),
by constructing a \emph{positively-homogeneous order-preserving continuous (PHOC)} mapping $G\colon \real_{+}^N\to\real_{+}^N$ from
$f_n~(n \in \mathcal{N})$ and $\|\cdot\|$
in \eqref{eq.C},
we can rewrite the set $\mathcal{S}$ in the following equivalent form \cite{renatoconvex}: 
\begin{align}
\label{eq.def.regionsinr}
\mathcal{S} = \{\signal{s}\in\real_{+}^N\mid \spradG (\bm{s}) \le 1\},
\end{align}
where, for each $\signal{s}\in\real_+^N$, $\spradG (\bm{s}) \in \real_{+}$ is
the \emph{spectral radius} (see Definition \ref{def.nl_radius})
of the (possibly nonlinear) PHOC mapping defined by
$$
\real_{+}^{N}\to\real_{+}^{N}\colon \signal{p}\mapsto \diagvec{s}(G(\bm{p})).
$$
Furthermore, by the nonlinear bijection
$$
\real_{+}^{N} \ni \signal{r} = (\log(1+s_n))_{n \in \mathcal{N}}\Leftrightarrow 
(e^{r_n}-1)_{n \in \mathcal{N}} = \signal{s} \in \real_{+}^{N},
$$
Problem \eqref{eq.sumrate.sinr} is equivalent to
\begin{align}
	\label{eq.sumrate.rate}\tag{R}
		\minimize_{\signal{r}\in\mathcal{R}}~ {}- \bm{w}^{\mathsf{T}}\signal{r},
\end{align}
where the \emph{achievable rate region} $\mathcal{R}$ is given by
\begin{align}
\label{eq.def.regionrate}
&\mathcal{R} := \{\signal{r}\in\real_{+}^N\mid 
(\spradG \circ \mathfrak{E})(\bm{r})  \leq 1 \},\\
\label{eq.def.VecExp}
	&\mathfrak{E} \colon \real_+^N\to\real_+^N\colon\bm{r}\mapsto (e^{r_n}-1)_{n \in \mathcal{N}},
\end{align}
and $\signal{w} := (w_n)_{n \in \mathcal{N}} \in \real_{++}^{N}$.
A practical sufficient condition for convexity of the nonlinear spectral radius function $\spradG\colon\real_{+}^{N}\to\real_{+}$ is revealed
in \cite{renatoconvex}, with a focus on PHOC mappings $G$ constructed from
UatF-based instances of $f_n~(n \in \mathcal{N})$ and \emph{polyhedral} order-preserving norms $\|\cdot\|$.
It is also shown that convexity of $\spradG$
implies convexity of $\spradG\circ\mathfrak{E}\colon\real_{+}^{N}\to\real_{+}$ \cite[Proposition 6]{renatoconvex}.
Both Problems \eqref{eq.sumrate.sinr} and \eqref{eq.sumrate.rate}
are guaranteed to be convex under the sufficient condition in \cite{renatoconvex}. However, in general, Problem \eqref{eq.sumrate.rate} is more likely to be convex
than Problem \eqref{eq.sumrate.sinr}
because $\spradG\circ\mathfrak{E}$ can be convex
even if $\spradG$ is nonconvex.
Hence, we focus on solving Problem \eqref{eq.sumrate.rate},
while Problem \eqref{eq.sumrate.sinr} is addressed in the Supplemental Material.
Note that, once a solution to Problem \eqref{eq.sumrate.rate} is obtained,
an efficient scheme for computing
a solution to Problem \eqref{eq.sumrate.power} is available
(see Appendix \ref{appendix.conversion} for details).

Even under the convexity condition, solving Problem \eqref{eq.sumrate.rate}
is challenging because the projection onto the constraint set $\mathcal{R}$ is difficult to compute.
In fact, this difficulty already arises even for the simplest linear PHOC mapping 
$G\colon\real_{+}^N\to\real_{+}^N\colon\signal{x}\mapsto\signal{M}\signal{x},$
where $\signal{M} \in \real_{+}^{N \times N}$ is a given nonnegative matrix.
In this special case, for every $\bm{s} \in \real_+^N$,
$\spradG(\bm{s})$ coincides with the largest absolute eigenvalue of $\diagvec{s}\signal{M} \in \real_{+}^{N \times N}$, i.e., the usual notion of the spectral radius of $\diagvec{s}\signal{M}$ in linear algebra.
However, typically $\diagvec{s}\signal{M}$ is \emph{nonsymmetric} for any given $\bm{s} \in \real_{+}^{N} \setminus \{ \signal{0}\}$, so 
computing the projection onto $\mathcal{R}$ is challenging. Indeed, projection formulas onto level sets of the spectral radius function are available only for \emph{symmetric} matrices, and the derivations rely crucially on the fact that every symmetric matrix admits an orthogonal diagonalization \cite[Corollary~24.65]{baus17}. In contrast, a nonsymmetric matrix need not be diagonalizable, so  explicit projection formulas are generally unavailable outside the symmetric setting. A more challenging situation arises in nonlinear settings, such as when projecting onto level sets of the nonlinear spectral radius function $\spradG$. This challenge is hard to overcome even with sophisticated splitting schemes such as those in \cite{combettes11,combettes21,condat23,chierchia14,chierchia15,wang16,kyochi21} (see also Sect.~\ref{sect.intro.related}).

\subsection{Main contributions}
The first objective of this study is to construct an efficient
algorithm that provably converges to the global optimal value of
Problem \eqref{eq.sumrate.rate}
-- and hence that of the original problem \eqref{eq.sumrate.power} --
under the assumption that $\spradG\circ\mathfrak{E}\colon\real_{+}^{N}\to\real_{+}$ is convex.
To this end, we exploit the \emph{subgradient projection} as a key ingredient to solve Problem \eqref{eq.sumrate.rate}.
More precisely, 
after introducing necessary mathematical notions in Sect.~\ref{sect.mathpre},
in Sect.~\ref{sect.reformulation}, we carefully reformulate Problem \eqref{eq.sumrate.rate} to enable the use of algorithmic frameworks based on subgradient projections.
One challenge to overcome is that existing frameworks rely on subgradient projection mappings relative to convex functions that are real-valued on vector spaces. As a result, these frameworks are not directly applicable to
Problem \eqref{eq.sumrate.rate} because the function $\spradG$ is well defined only on the nonnegative orthant $\real_{+}^{N}$ (see Definitions \ref{def.nl_radius} and \ref{def:spradG}).
We address this difficulty by carefully translating Problem \eqref{eq.sumrate.rate}
into an equivalent problem composed of functions from $\real^{N}$ to $\real$.
Then, in Sect.~\ref{sect.baseAlg},
we derive the proposed algorithm by applying the hybrid steepest descent (HSD) method \cite{yamada04b}
to the reformulated problem,
and then we prove its convergence to the global optimum of
Problem \eqref{eq.sumrate.rate} under the assumption of convexity.
In Sect.~\ref{sect.concAlg}, by focusing on typical models (Example \ref{ex.maxlinearG}) that appear in cell-less networks,
we show that the steps of the proposed algorithm can be efficiently computed at each iteration.

The second objective of this study is to illustrate the effectiveness of the proposed algorithm for UatF-based resource allocation in cell-less networks. In Sect.~\ref{sect.wireless.system}, we validate the models used to derive the algorithm and the associated theoretical results. In Sect.~\ref{sect.wireless.simu}, we verify conditions for convexity of Problem~\eqref{eq.sumrate.rate} and confirm our theoretical guarantee that the proposed algorithm converges to global optima. Moreover, the results indicate that the proposed algorithm achieves convergence speed comparable to that of the WMMSE method, a representative fast heuristic that may not converge to global optima even if Problem~\eqref{eq.sumrate.rate} is convex.

\subsection{Relation to existing studies on optimization}
\label{sect.intro.related}
\subsubsection*{Constrained convex optimization}
For convex optimization problems where the projections onto constraint sets are difficult to compute,
subgradient projections can be exploited as low-complexity alternatives.
Since subgradient projection mappings are \emph{quasi-nonexpansive},
a promising approach is to reformulate the problem of interest
as minimizing a differentiable convex function
over the fixed point set of a quasi-nonexpansive mapping (see Sect.~\ref{sect.pre.convex} for details). This reformulation results in a problem that can be efficiently solved with the HSD method\footnote{%
The HSD method was originally proposed for variational inequality problems over the fixed point sets of nonexpansive mappings in \cite{yamada01} and was subsequently generalized for quasi-nonexpansive mappings in \cite{yamada04b}.}.
Indeed, the HSD method has been successfully applied to
problems where only the subgradient projections onto constraint sets are computable: minimal antenna-subset selection \cite{yamada11} and image recovery under non-Gaussian noise \cite{ono15}.
While the algorithm we propose is also based on the HSD method,
the problems considered in \cite{yamada11,ono15}
differ substantially from
Problem \eqref{eq.sumrate.rate} (and also Problem \eqref{eq.sumrate.sinr}). As a result, significant new contributions are needed
to construct the proposed algorithm and to prove
its convergence to the global optimal value of Problem \eqref{eq.sumrate.rate}.

Generalized Haugazeau's method \cite{combettes00,combettes03} is designed
to minimize a strictly convex function over the fixed point set of a quasi-nonexpansive mapping.
Since the cost function in Problem \eqref{eq.sumrate.rate} is not strictly convex,
generalized Haugazeau's method is not applicable to Problem \eqref{eq.sumrate.rate}.
Although the cost function in Problem \eqref{eq.sumrate.sinr} is strictly convex, the subproblem that needs to be solved at each iteration of  generalized Haugazeau's method seems difficult to solve in our setting.

Proximal splitting methods \cite{combettes11,combettes21,condat23} are widely used for constrained and/or nonsmooth convex optimization problems
because of the ability to address complex problems by splitting them into tractable components 
\cite{gandy11,puste11,combettes20,kuroda22,condat22,kuroda25}.
However, these methods are not suitable for Problems \eqref{eq.sumrate.rate} and \eqref{eq.sumrate.sinr}
because computing the projections onto level sets of functions involving the nonlinear spectral radius function $\spradG$ is challenging (see Sect.~\ref{sect.intro.reformulation}),
and it is also difficult to decompose $\spradG$ into simpler functions
that can be handled by these methods.

Epigraphical projection techniques \cite{chierchia14,chierchia15,wang16,kyochi21}
are also used to deal with complicated constraint sets.
However, these techniques solve
relaxed problems that are not necessarily the same as the original problems.
In addition, these reformulations of Problems \eqref{eq.sumrate.rate} and \eqref{eq.sumrate.sinr} 
remain difficult to solve because computing the projection onto the epigraph of $\spradG$ is challenging.

\subsubsection*{Matrix eigenvalue optimization}
If we restrict our attention to linear PHOC mappings
$G\colon\real_{+}^N\to\real_{+}^N\colon\signal{x}\mapsto\signal{M}\signal{x}$
with $\signal{M} \in \real_{+}^{N \times N}$,
Problems \eqref{eq.sumrate.rate} and \eqref{eq.sumrate.sinr}
are related to matrix eigenvalue optimization
but have not been addressed in the literature.
While convex optimization involving eigenvalues of symmetric matrices is well studied \cite{overton88,vande96,lewis03,journ10,borwein13,douik19,benfe20},
these results are difficult to exploit for Problems \eqref{eq.sumrate.rate} and \eqref{eq.sumrate.sinr}. Indeed, for the applications under consideration,
the matrix $\signal{M}$ is nonsymmetric in general, so the matrix $\diagvec{s}\signal{M}$ is also nonsymmetric for any given $\bm{s} \in \real_{+}^{N}\setminus\{\signal{0}\}$, except for specific combinations of $\signal{M}$ and $\signal{s}$ that are unlikely to arise.
Note that, although convex optimization involving singular values of general matrices is well studied, singular values and eigenvalues are distinct notions for nonsymmetric matrices.
Since optimization problems involving the spectral radius of
(possibly nonsymmetric) nonnegative matrices are notoriously challenging,
only specific classes of problems have been addressed in the literature:
minimization or maximization of only the spectral radius over the so-called product family of nonnegative matrices \cite{blondel10,nesterov13,protasov16,cvetkovic22} and minimization of the max norm under the spectral radius constraint on nonnegative matrices \cite{nesterov20}.
The approaches in \cite{blondel10,nesterov13,protasov16,nesterov20,cvetkovic22} are based on the special properties of their problems
that are substantially different from Problems \eqref{eq.sumrate.rate} and \eqref{eq.sumrate.sinr}
even for linear PHOC mappings. As a result, these approaches are difficult to apply to the problems considered here.

\section{Mathematical preliminaries}
\label{sect.mathpre}
\subsection{Notation and definitions}
\label{sect.notation}
This subsection focuses on unifying terminology across various mathematical and wireless research domains.
We use the convention that the set $\Natural$ of natural numbers does not include zero.
The set of nonnegative and positive reals are denoted by $\real_{+}:= [0,\infty)$ and $\real_{++}:= (0,\infty)$, respectively.
We use the notation $\mathcal{N} := \{1,\ldots,N\}$ with $N \in \Natural$.
For a given vector $\signal{x} \in \real^N$,
its $n$th coordinate is denoted as $x_n$ for every $n \in \mathcal{N}$.
Inequalities involving vectors should be understood coordinatewise.
For example, given $\signal{x} \in \real^N$ and
$\signal{y} \in \real^N$,
we write $\signal{x} \leq \signal{y}$
if and only if $(\forall n \in \mathcal{N})~x_n \leq y_n$.
We write the transpose and the Hermitian transpose of a matrix $\signal{M}$ as $\signal{M}^{\mathsf{T}}$
and $\signal{M}^{\mathsf{H}}$, respectively.
We denote the diagonal matrix with the entries of $\signal{x} \in \mathbb{R}^{N}$ on the main diagonal
by $\mathrm{diag}(\signal{x}) \in \mathbb{R}^{N \times N}$.

A sequence $(\signal{x}_n)_{n\in\Natural} \subset \real^N$ is said to converge to $\signal{x} \in \real^N$ if $\lim_{n\to\infty}\|\signal{x}_n-\signal{x}\|=0$ for some (and, hence, every) norm $\|\cdot\|$ on $\real^N$.
Let $\mathcal{X}\subset\real^N$ and $\mathcal{Y}\subset\real^M$.
A mapping $f\colon\mathcal{X}\to\mathcal{Y}$ is said to be continuous if, for every sequence $(\signal{x}_n)_{n\in\Natural}$ in $\mathcal{X}$ converging to $\signal{x}\in\mathcal{X}$, the sequence $(f(\signal{x}_n))_{n\in\Natural}$ in $\mathcal{Y}$ converges to $f(\signal{x})\in\mathcal{Y}$.
A set $S \subset \real^N$ is said to be closed
if the limit of every convergent sequence $(\signal{x}_n)_{n\in\Natural}$ in $S$
lies in $S$;
bounded if $(\exists R > 0)(\forall \bm{x} \in S)~\|\bm{x}\| \leq R$ for some (and, hence, every) norm $\|\cdot\|$ on $\real^N$; and compact if and only if $S$ is bounded and closed (Heine-Borel theorem).

The ($1$-lower) level set of
$f\colon \real^{N}\to\real$ is denoted by
$\mathrm{lev}_{\le 1}(f) := \{\signal{x} \in \real^{N} \mid f(\signal{x}) \le 1\}$.
Let $C\subset\real^N$ be a nonempty convex set.
A function $f\colon C\to\real$ is \emph{convex} if $f(\alpha \signal{x}+(1-\alpha)\signal{y})\le \alpha f(\signal{x})+(1-\alpha)f(\signal{y})$
for every $\signal{x}$ and $\signal{y}$ in $C$ and $\alpha \in (0,1)$;
and \emph{strictly convex} if $f(\alpha \signal{x}+(1-\alpha)\signal{y}) < \alpha f(\signal{x})+(1-\alpha)f(\signal{y})$ for every $\signal{x}$ and $\signal{y}$ in $C$ and $\alpha \in (0,1)$.
The set of fixed points of a mapping $T\colon C \to C$ is denoted by $\mathrm{Fix}(T):=\{\signal{x}\in C \mid T(\signal{x})=\signal{x}\}$.

A norm $\|\cdot\|$ on $\real^N$ is said to be \emph{order-preserving} if\linebreak[4] $(\forall\signal{x}\in\real_+^N) (\forall\signal{y}\in\real_+^N)~\signal{x}\le\signal{y}\Rightarrow \|\signal{x}\|\le\|\signal{y}\|$; and
\emph{polyhedral} if there exist $L\in\Natural$ and $\signal{a}_1,\ldots,\signal{a}_L$ in $\real^N_{+}\setminus\{\signal{0}\}$ such that $(\forall\signal{x}\in\real^N)~\|\signal{x}\|=\max_{l\in\{1,\ldots,L\}}\signal{a}_l^{\mathsf{T}}|\signal{x}|$, where 
$|\signal{x}| := (|x_n|)_{n \in \mathcal{N}}\in\real^N_+$.

A function $f\colon \real_+^N\to\real_{+}$ is said to be \emph{order-preserving} if $(\forall\signal{x}\in\real_+^N)(\forall\signal{y}\in\real_+^N)~\signal{x}\le\signal{y}\Rightarrow f(\signal{x})\le f(\signal{y})$;
\emph{positively homogeneous} if $(\forall\signal{x}\in\real_+^N)(\forall\alpha>0)~ f(\alpha\signal{x}) = \alpha f(\signal{x})$;
and \emph{strictly subhomogeneous} if $(\forall\signal{x}\in\real_+^N)(\forall\alpha \in (0, 1))~ f(\alpha\signal{x}) > \alpha f(\signal{x})$.
Differently from \cite{lem13,krause2015positive}, we do not exclude $\bm{x} = \bm{0}$
in our definition of strict subhomogeneity, which can be shown  to be equivalent to \emph{scalability}
$(\forall\signal{x}\in\real_+^N)(\forall\beta > 1)~ f(\beta\signal{x}) < \beta f(\signal{x})$
studied in the wireless literature \cite{yates95,martin11}.

	A mapping $F\colon \real_+^N\to\real_{+}^N\colon \signal{x}\mapsto
	(f_n(\signal{x}))_{n \in \mathcal{N}}$ is called
	a SSOC mapping
	(respectively, PHOC mapping)
	if $f_n\colon \real_+^N\to\real_{+}$ is strictly subhomogeneous, order-preserving, and continuous
	(respectively, positively homogeneous, order-preserving, and continuous) for every
	$n \in \mathcal{N}$.
If $F\colon \real_{+}^{N}\to\real_{+}^{N}$ is strictly subhomogeneous and order-preserving,
then we have $(\forall \bm{x}\in\real_+^N)~F(\bm{x}) > \signal{0}$, and hence
we usually denote a SSOC mapping as $F\colon \real_{+}^{N}\to\real_{++}^{N}$.

The spectral radius of PHOC mappings is defined as the supremum of all eigenvalues \cite[Definition~3.2]{nussbaum1986convexity}: 

\begin{definition}
	\label{def.nl_radius} The spectral radius of a PHOC mapping $G\colon \real^N_+\to\real^N_+$ is defined to be 
	\begin{align}
		\label{eq.nl_radius}
		\rho(G) := \sup\{\lambda\in\real_+~|~(\exists \signal{x}\in\real_+^N\setminus\{\signal{0}\})~ G(\signal{x})=\lambda\signal{x}\}.
	\end{align}
\end{definition}
We recall that there always exists a nonnegative scalar and a corresponding vector in $\real_+^{N}\setminus\{\signal{0}\}$ attaining the supremum in \eqref{eq.nl_radius} \cite[Proposition~5.3.2(ii) and Corollary~5.4.2]{lem13}.
For a linear PHOC mapping $G\colon \real_{+}^{N}\rightarrow\real_{+}^{N}\colon \signal{x}\mapsto\bm{M}\bm{x}$
with $\signal{M} \in \real_{+}^{N \times N}$,
standard Perron-Frobenius theory shows that
$\rho(G)$ is identical to the usual notion of the spectral radius of $\bm{M}$ in linear algebra.
The nonlinear spectral radius function $\spradG$
used in \eqref{eq.def.regionsinr} and \eqref{eq.def.regionrate} is formally defined as follows.

\begin{definition}
\label{def:spradG}
For a given PHOC mapping $G\colon \real_{+}^{N}\to\real_{+}^N$,
we define the function
$$\spradG \colon \real^N_+\to\real_+\colon \signal{x}\mapsto\rho(T_{G,\bm{x}}),$$
where, for each $\bm{x} \in \real_+^N$, we recall that $\rho(T_{G,\bm{x}})$ is
the spectral radius of the PHOC mapping defined by
$$
T_{G,\bm{x}}\colon\real_{+}^{N}\to\real_{+}^{N}\colon \signal{p}\mapsto \diagvec{x}(G(\bm{p})).
$$
\end{definition}
For a linear PHOC mapping $G\colon \real_{+}^{N}\rightarrow\real_{+}^{N}\colon \signal{x}\mapsto\bm{M}\bm{x}$
with $\signal{M} \in \real_{+}^{N \times N}$,
by regarding $\bm{M}$ as this linear mapping (with the standard basis),
we also write $\varrho_{\bm{M}}$ instead of $\spradG$.

\subsection{Convex analysis and fixed point theory in Euclidean spaces}
\label{sect.pre.convex}
This subsection presents necessary results of convex analysis and fixed point theory
in a Euclidean space $\real^{N}$
equipped with the standard inner product
$\langle \bm{x}, \bm{y} \rangle := \bm{x}^{\mathsf{T}}\bm{y}$ for every $(\bm{x},\bm{y}) \in \real^N \times \real^N$
and its induced norm $\|\bm{x}\|_2 := \sqrt{\langle \bm{x}, \bm{x} \rangle }$ for every $\bm{x} \in \real^N$.
A mapping $T\colon \real^N\to\real^N$ is called \emph{Lipschitz continuous} on a set $S \subset \real^{N}$ with Lipschitz constant $\kappa > 0$ if
\begin{align*}
(\forall \bm{x} \in S)(\forall \bm{y} \in S)\quad\|T(\bm{x})-T(\bm{y}) \|_{2} \leq \kappa\|\bm{x}-\bm{y}\|_{2}.
\end{align*}
In particular, $T\colon \real^N\rightarrow\real^N$ is called \emph{nonexpansive} if it is Lipschitz continuous on $\real^N$ with Lipschitz constant $1$.
A mapping $T\colon \real^N\rightarrow\real^N$ is called \emph{quasi-nonexpansive} if
\begin{align*}
(\forall \bm{x}\in\real^N)(\forall \bm{y}\in\mathrm{Fix}(T))\quad\|T(\bm{x})-\bm{y} \|_{2} \leq \|\bm{x}-\bm{y}\|_{2}.
\end{align*}
Let $C \subset \mathbb{R}^{N}$ be a nonempty closed convex set.
For every $\signal{x} \in \real^{N}$, there exists a
unique point $P_{C}(\signal{x}) \in C$ satisfying
\begin{align}
\label{eq.def.EucDist}
d_2(\signal{x},C) := \min_{\signal{y} \in C}\|\bm{x}-\bm{y} \|_2 = \|\bm{x}-P_{C}(\signal{x}) \|_2,
\end{align}
and $P_{C}\colon\real^{N}\to C$ is called the \emph{projection mapping} onto $C$.
The \emph{subdifferential} of a convex function $h \colon \real^{N}\to\real$ at $\bm{x}\in\mathbb{R}^{N}$ is defined by
\begin{align*}
\partial h(\bm{x}) :=  \{\bm{u} \in \mathbb{R}^{N} \,|\, (\forall \bm{y} \in \mathbb{R}^{N})~\langle\bm{y}-\bm{x},\bm{u} \rangle + h(\bm{x}) \leq h(\bm{y}) \},
\end{align*}
and $\bm{u} \in \partial h(\bm{x})$ is called a \emph{subgradient} of $h$ at $\signal{x}$.
Convex functions from $\real^{N}$ to $\real$ have the following useful property.
\begin{fact}[{\cite[Proposition 16.20]{baus17}}]
\label{fact.boundedSubgrad}
Let $h\colon \real^{N}\to\real$ be a convex function. Then
$h$ is continuous on $\real^N$, and $\partial h(\bm{x}) \neq \varnothing$ for every $\bm{x} \in \real^{N}$.
Moreover, for every bounded subset $S$ of $\real^{N}$, the set $\{\bm{u} \in \partial h(\bm{x}) \mid \bm{x} \in S \}$ is bounded.
\end{fact}

By Fact \ref{fact.boundedSubgrad} and Fermat's rule $\bm{x}^{\star} \in  \argmin_{\bm{x}\in\mathbb{R}^{N}} h(\bm{x}) \Leftrightarrow \bm{0} \in \partial h(\bm{x}^{\star})$,
subgradient projection mappings can be defined as follows.
\begin{definition}
\label{def.subgradProj}
Let $h\colon \real^{N}\to\real$ be a convex function with
$\mathrm{lev}_{\le 1}(h) \neq \varnothing$.
Let $g\colon \real^{N}\to\real^{N}$ be a selection of the subdifferential of $h$, i.e.,
$(\forall \signal{x} \in \real^{N})~g(\signal{x}) \in \partial h(\signal{x})$.
The subgradient projection mapping
$T_{\mathrm{sp}(h)}\colon \real^{N}\to\real^{N}$, relative to $h$, onto $\mathrm{lev}_{\le 1}(h)$ is defined by
	\begin{align*}
		T_{\mathrm{sp}(h)}(\signal{x}) :=
		\begin{cases}
			 \signal{x} - (h(\signal{x})-1)g(\signal{x})/\|g(\signal{x})\|_2^2,
			 & \text{if}~ h(\signal{x}) > 1;\\
			 \signal{x},
			 & \text{if}~h(\signal{x}) \le 1.
		\end{cases}
	\end{align*}
\end{definition}

The following result establishes the convergence of the HSD method using subgradient projections
for a class of constrained convex optimization problems.
\begin{fact}[HSD method using subgradient projections {\cite{yamada04b}}]
\label{fact.hsdmsp}
Suppose that the following conditions hold:
\begin{enumerate}[\itemsep=2pt]
\item[(a1)] $h\colon \real^{N}\to\real$ is convex, $\mathcal{K} \subset \real^{N}$ is compact and convex, and $\mathcal{K} \cap \mathrm{lev}_{\le 1}(h) \neq \varnothing$; and
\item[(a2)] $\Theta: \real^{N}\to\real$ is convex and (G\^{a}teaux) differentiable (see Definition \ref{def.differentiable} in Appendix \ref{appendix.proof.subgradient})
on $\mathbb{R}^{N}$, and its
gradient $\nabla \Theta:\mathbb{R}^{N}\to\mathbb{R}^{N}$ is Lipschitz continuous on $\mathcal{K}$.
\end{enumerate}
Then we have\footnote{%
The mapping $P_{\mathcal{K}} \circ T_{\mathrm{sp}(h)}$ belongs to a class of \emph{quasi-shrinking mappings} \cite[Definition 1]{yamada04b} -- a special subclass of quasi-nonexpansive mappings.}
$$\mathcal{S} := \argmin_{\signal{x} \in \mathcal{K} \cap \mathrm{lev}_{\le 1}(h)}\Theta(\signal{x})
= \argmin_{\signal{x} \in \mathrm{Fix}(P_{\mathcal{K}} \circ T_{\mathrm{sp}(h)})}\Theta(\signal{x})
\neq \varnothing.
$$
Furthermore, let $\bm{x}_1 \in \real^{N}$ and $(\mu_k)_{k\in\mathbb{N}} \subset \real_{+}$ such that
$\lim_{k\to\infty}\mu_k = 0$ and $\sum_{k\in\mathbb{N}}\mu_k = \infty$,
and generate the sequence
$(\signal{x}_k)_{k\in\mathbb{N}}$ by
	\begin{align*}
		\signal{x}_{k+1}
		:= (P_{\mathcal{K}} \circ T_{\mathrm{sp}(h)})(\signal{x}_k)
		-\mu_{k}\nabla\Theta ((P_{\mathcal{K}} \circ T_{\mathrm{sp}(h)})(\signal{x}_k))
	\end{align*}
for every $k \in \mathbb{N}$. Then $\lim_{k\to\infty}d_2(\signal{x}_k,\mathcal{S}) = 0$ holds.
\end{fact}
\begin{remark}
In \cite[Proposition 6]{yamada04b},
it is stated as an assumption that
any selection $g\colon \real^{N}\to\real^{N}$ of $\partial h$
(i.e., that satisfies $(\forall \signal{x} \in \real^{N})~ g(\signal{x}) \in \partial h (\signal{x})$)
is bounded on every bounded subset of $\real^N$.
Since we focus on the finite-dimensional space $\real^{N}$, this condition automatically holds by Fact \ref{fact.boundedSubgrad}.
\end{remark}

\subsection{Nonlinear spectral radius functions for characterizing achievable rate regions}
\label{sect.general.sumrateprob}
This subsection collects the essential properties of
the PHOC mapping $G$
and the nonlinear spectral radius function $\spradG$ used to characterize the achievable rate region $\mathcal{R}$ in Problem \eqref{eq.sumrate.rate}.
In particular, we focus on the following class of PHOC mappings
$\fnorm\colon \real_{+}^{N}\rightarrow\real_{+}^N$
constructed from
SSOC mappings $F$ and order-preserving norms $\|\cdot\|$ by the following procedure given in \cite{renatoconvex}
(see Sect.~\ref{sect.wireless} for its relation to wireless applications).

\begin{assumption}
\label{asmp.compatible}
Let $F\colon \real_+^N\to\real_{++}^N\colon \signal{x}\mapsto (f_n(\signal{x}))_{n \in \mathcal{N}}$ be a continuous mapping.
For each $n \in \mathcal{N}$,
we assume that there exist a nonempty set
$\mathcal{Y}_n$, positive scalars $(u_{n, y})_{y \in \mathcal{Y}_n}$ bounded away from zero, and PHOC functions $(g_{n, y})_{y \in\mathcal{Y}_n}$, such that
	$$(\forall\signal{x}\in\real_+^N)\quad f_n(\signal{x}) =\inf_{y\in\mathcal{Y}_n}(g_{n, y}(\signal{x})+{u}_{n, y}).$$
    In addition, given an order-preserving norm $\|\cdot\|$ on $\real^N$, we assume that the function $(f_n)_{\|\cdot\|}\colon\real_{+}^{N}\to\real_{+}$ defined by
    $$(\forall\signal{x}\in\real_+^N)\quad (f_n)_{\|\cdot\|}(\signal{x}) := \inf_{y\in\mathcal{Y}_n}(g_{n, y}(\signal{x})+{u}_{n, y}\|\signal{x}\|)$$
    is continuous on $\real_{+}^N$ for every $n \in \mathcal{N}$. (NOTE: we assume everywhere continuity to avoid technical digressions.)
    \end{assumption}

\begin{fact}[{\cite[Lemma 1]{renatoconvex}}]
	\label{fact.compatible}
	Under Assumption \ref{asmp.compatible}, the mapping $F$ is a SSOC mapping, and the mapping
    $$F_{\|\cdot\|}\colon \real^N_+\to\real_{+}^N\colon\signal{x}\mapsto((f_n)_{\|\cdot\|}(\signal{x}))_{n\in\mathcal{N}}$$
	is a PHOC mapping.
\end{fact}

\begin{remark}
If $f_n~(n \in \mathcal{N})$ and $\|\cdot\|$ in Problem \eqref{eq.sumrate.power} satisfy Assumption \ref{asmp.compatible}, then the global optimal value of Problem \eqref{eq.sumrate.power} is the same as that of Problem \eqref{eq.sumrate.rate} with $G = \fnorm$ (see Appendix \ref{appendix.conversion} for the scheme for converting solutions).
\end{remark}

The next fact shows properties of 
the nonlinear spectral radius function $\spradG$
for the class of PHOC mappings $G = \fnorm$.

\begin{fact}[{\cite[Proposition 2]{renatoconvex}}]
	\label{fact.qnorm}
	Suppose that $G =\fnorm$ holds for some PHOC mapping $\fnorm\colon \real_+^N\to\real_{+}^N$ constructed from $F$ and $\|\cdot\|$ that satisfy Assumption \ref{asmp.compatible}.
	Then, for the function $\spradG$ in Definition \ref{def:spradG},
	 each of the following holds:
	\begin{enumerate}
		\item[(i)] $(\forall \alpha>0)(\forall \signal{x}\in\real^N_+)\quad \spradG(\alpha\signal{x})=\alpha \spradG(\signal{x})$\vspace{2pt};
		\item[(ii)] $(\forall \signal{x}\in\real^N_+)\quad\spradG(\signal{x})=0\Leftrightarrow \signal{x}=\signal{0}$\vspace{1pt};
		\item[(iii)] $\spradG$ is order-preserving; and\vspace{1pt}
		\item[(iv)] $\spradG$ is continuous on $\real^N_+$.
	\end{enumerate}
\end{fact}

The following example provides typical instances of $\fnorm$ constructed from $F$ and $\|\cdot\|$ that appear in UatF-based instances of Problem \eqref{eq.sumrate.power} for cell-less networks (see Sect.~\ref{sect.wireless} for application perspective).

\begin{example}
\label{ex.maxlinearG}
Let $F\colon \real_+^N\to\real_{++}^N\colon \signal{x}\mapsto \signal{Mx}+\signal{u}$
for given $\signal{M}\in\real_{++}^{N\times N}$ and $\signal{u}\in \real_{++}^N$,
and $\|\cdot\|$ be a polyhedral order-preserving norm, i.e., there exist $L\in\Natural$ vectors  $\signal{a}_1,\ldots,\signal{a}_L$ in $\real^N_{+}\setminus\{\signal{0}\}$ satisfying $(\forall\signal{x}\in\real^N)~\|\signal{x}\|=\max_{l\in\{1,\ldots,L\}}\signal{a}_l^{\mathsf{T}}|\signal{x}|$, where 
$|\signal{x}| = (|x_n|)_{n \in \mathcal{N}}$.
In this case, Assumption \ref{asmp.compatible} holds, and $\fnorm$ is given by $\fnorm\colon \real_+^N\to\real_{+}^N\colon \signal{x}\mapsto \signal{Mx}+\signal{u}\|\signal{x}\|.$ 
Let $G := \fnorm$ and $\signal{M}_{l} := \signal{M}+\signal{u}\signal{a}_l^{\mathsf{T}} \in \real_{++}^{N \times N}$ for every $l \in \{1,\ldots,L\}$.
Then, for the functions $\spradG$ and $\varrho_{\signal{M}_l}~(l = 1,\ldots,L)$ given in Definition \ref{def:spradG}, we have
		 \begin{align}
		 \label{eq.rho.maxlin}
		 (\forall \signal{x} \in \real_{+}^{N})\quad\spradG(\signal{x}) = \max_{l\in\{1,\ldots,L\}}  \varrho_{\signal{M}_l}(\signal{x}).
		 \end{align}
Note that $\spradG$ also satisfies the properties stated in Fact \ref{fact.qnorm}(i)--(iv).
The mathematical claims in this example
can be proved by following the argument of \cite[Example 5]{renatoconvex}.
\end{example}

\begin{fact}[{\cite[Example 4]{renatoconvex}}]
\label{fact.linearG}
In the setting of Example \ref{ex.maxlinearG},
for every $l \in \{1,\ldots,L\}$, $G_l\colon \real_+^N\to\real_{+}^N\colon \signal{x}\mapsto \signal{M}_l\signal{x}$
satisfies the assumptions in Fact \ref{fact.qnorm}, and
$\varrho_{\signal{M}_l}$ satisfies the properties shown in Fact \ref{fact.qnorm}(i)--(iv).
\end{fact}

\section{Proposed method}
\label{sect.optimization}
In this section, we present the proposed algorithm
that provably converges to the global optimal value of
Problem \eqref{eq.sumrate.rate} if the composite function $\spradG\circ\mathfrak{E}\colon\real_{+}^{N}\to\real_{+}$ is convex.
Based on the discussion in Sect.~\ref{sect.general.sumrateprob},
we focus on the following class of PHOC mappings $G$ in Problem \eqref{eq.sumrate.rate}.

\begin{assumption}
\label{asmp.G.hypTnorm}
We assume that $G =\fnorm$ holds for some PHOC mapping $\fnorm\colon \real_+^N\to\real_{+}^N$ constructed as shown in Fact \ref{fact.compatible}
from a SSOC mapping $F$ and an order-preserving norm $\|\cdot\|$ that satisfy Assumption \ref{asmp.compatible}.
\end{assumption}
In Sect.~\ref{sect.reformulation},
we reformulate Problem \eqref{eq.sumrate.rate}
involving functions well defined only on $\real_{+}^{N}$
to an equivalent optimization problem composed of functions from $\real^{N}$ to $\real$.
Then, in Sect.~\ref{sect.baseAlg},
we derive the proposed algorithm by applying the HSD method
to the reformulated problem, and then prove its convergence to
the global optimum of Problem \eqref{eq.sumrate.rate}
under the convexity condition.
In Sect.~\ref{sect.concAlg},
by focusing on the models given in Example \ref{ex.maxlinearG},
we present a low-complexity method for computing the subgradient projection used in the proposed algorithm.

\subsection{Reformulation to a problem over Euclidean spaces}
\label{sect.reformulation}
The nonlinear spectral radius function $\spradG$ is well defined only on $\real_{+}^{N}$ because of Definitions \ref{def.nl_radius} and \ref{def:spradG}.
As a result,
most convex optimization frameworks, including the HSD method,
are not directly applicable to Problem \eqref{eq.sumrate.rate} since these frameworks are designed for functions defined on the entire vector space.
Straightforward extensions of the functions in Problem \eqref{eq.sumrate.rate} to functions taking the extended value $+\infty$
on $\real^{N} \setminus \real_{+}^{N}$
are not suitable to address Problem \eqref{eq.sumrate.rate} because
the subgradient projection mappings are guaranteed to be well defined only
for convex functions taking real values on the whole space $\real^{N}$
(see Definition \ref{def.subgradProj} and Fact \ref{fact.boundedSubgrad}).
Real-valued convex extensions are also necessary to
satisfy the assumptions in Fact \ref{fact.hsdmsp} that shows convergence of the HSD method with subgradient projections.
Note that, to apply Fact \ref{fact.hsdmsp}, we need to translate
Problem \eqref{eq.sumrate.rate} into an equivalent problem that satisfies
all assumptions in Fact \ref{fact.hsdmsp}.

We first show that the composite function
$\spradG \circ \mathfrak{E} \colon \real^N_+\to\real_+$ can be extended
to a convex function from $\real^{N}$ to $\real$ under the following assumption
(see Sect.~\ref{sect.wireless} for the validity of this assumption in wireless applications):

\begin{assumption}
\label{asmp.convex.extension.rho.exp}
	In addition to Assumption \ref{asmp.G.hypTnorm},
	we assume that $\spradG \circ \mathfrak{E}\colon \real_{+}^{N}\to\real_{+}$ is convex on $\real_{++}^{N}$, where $\spradG$ and 
	$\mathfrak{E}$ are
	given in Definition \ref{def:spradG} and \eqref{eq.def.VecExp}, respectively.
\end{assumption}

\begin{lemma}
\label{lemma.convext.rhoexp}
Define
$$
h_{G, \mathfrak{E}} \colon \real^{N}\to\real_{+}\colon\bm{r}\mapsto (\spradG \circ \mathfrak{E} \circ P_{\mathbb{R}_{+}^{N}})(\bm{r}).
$$
Then, under Assumption \ref{asmp.convex.extension.rho.exp}, each of the following holds:
\begin{enumerate}[\itemsep=2pt]
\item[(i)] $(\forall \signal{r} \in \real_{+}^{N})\quad h_{G, \mathfrak{E}}(\signal{r}) = (\spradG \circ \mathfrak{E})(\signal{r})$; and
\item[(ii)] $h_{G, \mathfrak{E}}$ is convex on $\real^{N}$.
\end{enumerate}
\end{lemma}
\begin{proof}
(i) Clear from the definition of $h_{G, \mathfrak{E}}$.

(ii) We recall that Fact \ref{fact.qnorm} holds under Assumption \ref{asmp.G.hypTnorm} required by Assumption \ref{asmp.convex.extension.rho.exp}.
Since $\spradG \circ \mathfrak{E}$ is continuous on $\real_{+}^{N}$ by Fact \ref{fact.qnorm}(iv) and the definition of $\mathfrak{E}$ in \eqref{eq.def.VecExp}, convexity of $\spradG \circ \mathfrak{E}$
on $\real_{++}^{N}$ in Assumption \ref{asmp.convex.extension.rho.exp} implies its convexity on $\real_{+}^{N}$.
Choose $\signal{r}_1 \in \real^{N}$, $\signal{r}_2 \in \real^{N}$, and $\alpha \in (0, 1)$ arbitrarily.
By definition of $P_{\mathbb{R}_{+}^{N}}\colon \real^{N}\to\real_{+}^{N}$,
we have $P_{\mathbb{R}_{+}^{N}}(\signal{r}) = (\max \{0,  r_n\})_{n \in \mathcal{N}}$ for every $\signal{r} \in \real^{N}$. Hence, convexity of the function $\real\to\real\colon x \mapsto \max\{0, x\}$ implies convexity of
$P_{\mathbb{R}_{+}^{N}}$ in each coordinate:
\begin{align*}
P_{\mathbb{R}_{+}^{N}}(\alpha \signal{r}_1 + (1-\alpha)\signal{r}_2)
\le \alpha P_{\mathbb{R}_{+}^{N}}(\signal{r}_1) + (1-\alpha) P_{\mathbb{R}_{+}^{N}}(\signal{r}_2).
\end{align*}
Meanwhile, since $\spradG$ and $\mathfrak{E}$ are order-preserving by
Fact \ref{fact.qnorm}(iii) and by definition, respectively, their composition
$\spradG \circ \mathfrak{E}$ is also order-preserving, i.e.,
$$
(\forall \signal{x} \in \real_{+}^{N})(\forall \signal{y} \in \real_{+}^{N})~
\signal{x} \le \signal{y} 
\Rightarrow (\spradG \circ \mathfrak{E})(\signal{x}) \le (\spradG \circ \mathfrak{E})(\signal{y}).
$$
Altogether, convexity of $h_{G, \mathfrak{E}}$ on $\real^{N}$ is shown by
\begin{align*}
(\spradG &\circ \mathfrak{E})(P_{\mathbb{R}_{+}^{N}}(\alpha \signal{r}_1 + (1-\alpha)\signal{r}_2))\\
&\le (\spradG\circ \mathfrak{E}) (\alpha P_{\mathbb{R}_{+}^{N}}(\signal{r}_1) + (1-\alpha) P_{\mathbb{R}_{+}^{N}}(\signal{r}_2))\\
&\le \alpha (\spradG\circ \mathfrak{E})(P_{\mathbb{R}_{+}^{N}}(\signal{r}_1))
+(1-\alpha) (\spradG \circ \mathfrak{E})(P_{\mathbb{R}_{+}^{N}}(\signal{r}_2)),
\end{align*}
where the last inequality follows from convexity of $\spradG \circ \mathfrak{E}$ on $\real_{+}^{N}$.
\end{proof}

Meanwhile, the cost function $\bm{r}\mapsto-\signal{w}^{\mathsf{T}}\signal{r}$ in Problem \eqref{eq.sumrate.rate}
can be regarded as a function from $\real^{N}$ to $\real$.
We state the following trivial lemma for later reference.
\begin{lemma}
\label{lemma.LipGradCost.rate}
Let $\Psi\colon \real^{N}\to\real\colon \signal{r}\mapsto-\signal{w}^{\mathsf{T}}\signal{r}$. Then
$\Psi$ is convex and differentiable on $\real^{N}$. Its gradient is given by\linebreak[4]
$(\forall \signal{r} \in \real^{N})\quad\nabla \Psi(\signal{r}) = -\signal{w}$,
and hence it is Lipschitz continuous on $(\real^{N}, \|\cdot \|_2)$.
\end{lemma}

Under Assumption \ref{asmp.convex.extension.rho.exp},
using $h_{G, \mathfrak{E}}$ given in Lemma \ref{lemma.convext.rhoexp}
and $\Psi$ given in Lemma \ref{lemma.LipGradCost.rate}, we reformulate
Problem \eqref{eq.sumrate.rate} to
\begin{align}
	\label{eq.refsumrate.frrate}\tag{R'}
	 \begin{array}{cl}
		\text{minimize} & \Psi(\signal{r}) \\
		\text{subject~to} & \signal{r}\in [0, b]^{N}\cap
		\mathrm{lev}_{\le 1}(h_{G, \mathfrak{E}}),
	\end{array}
\end{align}
where we also replace $\real_{+}^{N}$ with $[0, b]^{N}$ using a sufficiently large constant $b > 0$
in order to avoid technical digressions of little practical relevance.
Under Assumption \ref{asmp.convex.extension.rho.exp}, if $b$ is large enough,
Problems \eqref{eq.sumrate.rate} and \eqref{eq.refsumrate.frrate}
are guaranteed to have the same nonempty solution set.
To prove this result, we use the following lemma.

\begin{lemma}
\label{lemm.bounded}
The set $\{\signal{x} \in \real_{+}^{N} \mid \spradG(\signal{x}) \leq 1\}$ is bounded
under Assumption \ref{asmp.G.hypTnorm}.
\end{lemma}
\begin{proof}
Assume for the sake of contradiction that the set $S_G :=\{\signal{x} \in \real_{+}^{N} \mid \spradG(\signal{x}) \leq 1\}$ is unbounded. In this case, there exists a sequence $(\signal{x}_n)_{n\in\Natural}$ in $S_G\setminus \{\signal{0}\}$ such that $\lim_{n\to\infty}\|\signal{x}_n\|_2=\infty$. For every $n\in\Natural$, define $\signal{d}_n:=\signal{x}_n/\|\signal{x}_n\|_2$. Since $(\signal{d}_n)_{n\in\Natural}$ is a bounded sequence, we can use the Bolzano-Weierstrass theorem to extract a convergent subsequence $(\signal{d}_n)_{n\in K\subset\Natural}$ from $(\signal{d}_n)_{n\in\Natural}$. Denote by $\signal{d}:=\lim_{n\in K} \signal{d}_n$ the limit of the (sub)sequence $(\signal{d}_n)_{n\in K}$, and note that $\|\signal{d}\|_2=1$ because,
for every $n \in K$, $\signal{d}_n$ belongs to the set $\{\signal{u}\in\real^N\mid\|\signal{u}\|_2=1\}$, which is closed by continuity of the norm.
 Therefore, in light of Fact \ref{fact.qnorm}(ii) and (iv), we deduce $\lim_{n\in K} \spradG(\signal{d}_n) = \spradG(\signal{d})>0$, which enables us to obtain the contradiction 
\begin{multline*}
1\overset{\text{(a)}} {\ge}\limsup_{n\in \Natural} \spradG(\signal{x}_n)\ge \limsup_{n\in K} \spradG(\|\signal{x}_n\|_2~\signal{d}_n) \\ \overset{\text{(b)}}{=} \limsup_{n\in K} \|\signal{x}_n\|_2~\spradG(\signal{d}_n) = \infty,
\end{multline*}
where (a) follows from $(\forall n\in\Natural)~\signal{x}_n\in S_G$ and (b) follows from the positive homogeneity of $\spradG$ [see Fact \ref{fact.qnorm}(i)].
\end{proof}

\begin{proposition}
\label{prop.existence.equivalence.prob.rate}
Under Assumption \ref{asmp.convex.extension.rho.exp},
the solution set of Problem \eqref{eq.refsumrate.frrate} is a nonempty compact convex set for any $b > 0$.
Moreover, if $b$ is large enough,
the solution sets of Problems \eqref{eq.sumrate.rate} and \eqref{eq.refsumrate.frrate}
are the same.
\end{proposition}
\begin{proof}
Fix $b > 0$ arbitrarily.
By setting $\bm{r} = \bm{0}$ in Lemma \ref{lemma.convext.rhoexp}(i), we have
$h_{G, \mathfrak{E}}(\signal{0})
= (\spradG \circ \mathfrak{E})(\signal{0}) = \spradG(\signal{0}) = 0$,
implying that the set $[0, b]^{N}\cap\mathrm{lev}_{\le 1}(h_{G, \mathfrak{E}})$ is nonempty.
Since $h_{G, \mathfrak{E}}\colon\real^N\to\real$ is convex by Lemma \ref{lemma.convext.rhoexp}(ii)
and hence continuous by Fact \ref{fact.boundedSubgrad}, the level set $\mathrm{lev}_{\le 1}(h_{G, \mathfrak{E}})$ is closed and convex.
Therefore, the set $[0, b]^{N}\cap\mathrm{lev}_{\le 1}(h_{G, \mathfrak{E}})$
is compact and convex.
Since $\Psi\colon\real^N\to\real$ given in Lemma \ref{lemma.LipGradCost.rate} is also convex, we deduce that
the solution set of Problem \eqref{eq.refsumrate.frrate} is a nonempty compact convex set.

Next, we show that
the solution sets of Problems \eqref{eq.sumrate.rate} and \eqref{eq.refsumrate.frrate}
are the same for sufficiently large $b$.
By definition of $\mathfrak{E}$ in \eqref{eq.def.VecExp},
if $\bm{r} \in \real_{+}^{N}$
and $\|\bm{r}\|_{2} \to \infty$,
we have $\mathfrak{E}(\bm{r}) \in \real_{+}^{N}$
and $\|\mathfrak{E}(\bm{r})\|_{2} \to \infty$.
This relation and Lemma \ref{lemm.bounded} imply that
$\mathcal{R}$ in \eqref{eq.def.regionrate} is bounded.
Thus,
since $\mathcal{R} \subset \real_{+}^{N}$ holds by definition,
there exists $c > 0$ such that
\begin{align*}
(\forall b \ge c)\quad \mathcal{R} = [0, b]^{N} \cap \mathcal{R}
= [0, b]^{N}\cap\mathrm{lev}_{\le 1}(h_{G,\mathfrak{E}}),
\end{align*}
where the last equality holds by Lemma \ref{lemma.convext.rhoexp}(i)
and the definition of $\mathcal{R}$.
Hence, for every $b \geq c$, the desired result follows from the definition of $\Psi$.
\end{proof}

\subsection{Optimization algorithm and its convergence guarantee}
\label{sect.baseAlg}
The reformulated problem \eqref{eq.refsumrate.frrate} is still challenging to solve because the projection onto the level set $\mathrm{lev}_{\le 1}(h_{G, \mathfrak{E}})$
is difficult to compute owing to the nonlinear spectral radius function
$\spradG$ involved in $h_{G, \mathfrak{E}}$. It is also difficult to decompose $\spradG$ into simpler functions that can be handled by the advanced splitting strategies reviewed in Sect.~\ref{sect.intro.related}.
Fortunately,
as shown in Sect.~\ref{sect.concAlg},
the subgradient projection mapping relative to $h_{G, \mathfrak{E}}$
can be computed efficiently in (pseudo) closed-form in the scenarios considered in Example \ref{ex.maxlinearG}.
Motivated by this result, we derive the proposed algorithm by applying the subgradient-projection-based HSD method shown in Fact \ref{fact.hsdmsp} to Problem \eqref{eq.refsumrate.frrate}.

Specifically, with an initial point $\signal{r}_{1} \in \real^{N}_{++}$
and a slowly decreasing step-size\footnote{%
We focus on a positive initial point and a positive step-size in order to derive an efficient subgradient computation method in Sect.~\ref{sect.concAlg}.} $(\mu_k)_{k\in\mathbb{N}} \subset \real_{++}^{N}$,
the proposed algorithm generates the sequence $(\signal{r}_{k})_{k \in \mathbb{N}}$ by
\begin{align}
\label{eq:HSDMmaxsumrate.rate}
\nonumber
&\text{For}~k = 1, 2, \ldots\\
&\left\lfloor\begin{aligned}
&\text{Choose}~\signal{g}_k \in \partial (h_{G, \mathfrak{E}})(\signal{r}_k),\\
&\hat{\signal{r}}_{k} := \begin{cases}
\bm{r}_{k} - (h_{G, \mathfrak{E}}(\signal{r}_k)-1)\signal{g}_k/\|\signal{g}_k\|_2^2,
& \text{if}~ h_{G, \mathfrak{E}}(\signal{r}_k) > 1;\\
\bm{r}_{k},
& \text{if}~ h_{G, \mathfrak{E}}(\signal{r}_k) \le 1,\\
\end{cases}\\
&\tilde{\signal{r}}_{k} := P_{[0,b]^{N}}(\hat{\signal{r}}_{k}),\\
&\signal{r}_{k+1} := \tilde{\signal{r}}_{k} - \mu_{k}\nabla\Psi(\tilde{\signal{r}}_{k}).
\end{aligned}\right.
\end{align}
Since $h_{G, \mathfrak{E}}$ is a convex function from $\real^{N}$ to $\real$ by Lemma \ref{lemma.convext.rhoexp},
the existence of $\signal{g}_k$ is guaranteed by Fact \ref{fact.boundedSubgrad} at each iteration $k \in \mathbb{N}$, while
a specific method for finding $\signal{g}_k$ is presented in Sect.~\ref{sect.concAlg}.
The projection mapping $P_{[0, b]^{N}}\colon \real^{N}\to[0, b]^{N}$ can be computed by
\begin{align}
\label{eq.def.projectionbox}
(\forall \signal{x} \in \real^{N})\quad P_{[0, b]^{N}}(\signal{x}) = (\min\{ \max\{0, x_n\}, b\})_{n \in \mathcal{N}},
\end{align}
and the expression for $\nabla\Psi$ is shown in Lemma \ref{lemma.LipGradCost.rate}.

In the following theorem, 
we show that, under Assumption \ref{asmp.convex.extension.rho.exp}, 
the distance between $(\signal{r}_{k})_{k \in \mathbb{N}}$ and the solution set of
Problem \eqref{eq.refsumrate.frrate} converges to zero,
and the weighted sum-rate $(\signal{w}^{\mathsf{T}} \signal{r}_{k})_{k \in \mathbb{N}}$ converges to (the negative of) the global optimal value of Problem \eqref{eq.sumrate.rate} --
and hence that of the original problem \eqref{eq.sumrate.power} --
if $b$ is large enough.

\begin{theorem}
\label{theo.converge.rate}
Let $\signal{w} \in \mathbb{R}_{++}^{N}$,
$b > 0$,
$\signal{r}_1 \in \real_{++}^{N}$, and
$(\mu_k)_{k\in\mathbb{N}} \subset \real_{++}^{N}$ satisfy
$\lim_{k\to\infty}\mu_k = 0$ and
$\sum_{k\in\mathbb{N}}\mu_k = \infty$.
Then, under Assumption \ref{asmp.convex.extension.rho.exp}, each of the following holds
for $(\signal{r}_{k})_{k \in \mathbb{N}}$ generated by \eqref{eq:HSDMmaxsumrate.rate}.
\begin{enumerate}[\itemsep=2pt]
\item[(i)] Let $\Omega := \argmin_{\signal{r} \in [0, b]^{N}\cap
		\mathrm{lev}_{\le 1}(h_{G, \mathfrak{E}})}\Psi(\signal{r})$. Then $\lim_{k \to \infty}d_2(\signal{r}_{k}, \Omega) = 0$;
\item[(ii)] $(\forall k \in \mathbb{N})\quad\signal{r}_{k} \in \real_{++}^{N}$;
\item[(iii)] $\lim_{k \rightarrow \infty} d_2(\signal{r}_{k}, \mathcal{R}) = 0$; and
\item[(iv)] If $b$ in Problem \eqref{eq.refsumrate.frrate} is large enough, then
$\lim_{k \to \infty}\signal{w}^{\mathsf{T}}\signal{r}_k = \max_{\signal{r} \in \mathcal{R}}\signal{w}^{\mathsf{T}}\signal{r}$.
\end{enumerate}
\end{theorem}
\begin{proof}
(i) To prove this claim via Fact \ref{fact.hsdmsp},
we show that conditions (a1) and (a2) in Fact \ref{fact.hsdmsp} are satisfied for
$h = h_{G, \mathfrak{E}}$, $\mathcal{K} = [0, b]^{N}$, and $\Theta = \Psi$ as follows.
For any $b > 0$, $[0, b]^{N}$ is a compact convex set.
By Lemma \ref{lemma.convext.rhoexp}, $h_{G, \mathfrak{E}}\colon \real^{N}\to\real$
is convex, and $[0, b]^{N}\cap\mathrm{lev}_{\le 1}(h_{G, \mathfrak{E}})$ is nonempty
(see the proof of Proposition \ref{prop.existence.equivalence.prob.rate}).
Hence, condition (a1) is verified.
Lemma \ref{lemma.LipGradCost.rate} ensures that condition (a2) is satisfied for $\Theta = \Psi$.

(ii) For every $k \in \mathbb{N}$, we have
$\tilde{\signal{r}}_{k} = P_{[0,b]^{N}}(\hat{\signal{r}}_{k}) \in \mathbb{R}_{+}^{N}$ in \eqref{eq:HSDMmaxsumrate.rate}. We also have
$(\forall \signal{r} \in \mathbb{R}^{N})~ -\nabla\Psi(\signal{r}) = \signal{w} \in \real_{++}^{N}$ and
$(\mu_k)_{k\in\mathbb{N}} \subset \real_{++}^{N}$ by assumption.
Hence, the last step in \eqref{eq:HSDMmaxsumrate.rate} ensures
$\signal{r}_{k+1} \in \real_{++}^{N}$ for every $k \in \mathbb{N}$.
Note that $\signal{r}_{1} \in \real_{++}^{N}$ holds by assumption.

(iii) Since $\mathcal{R} = \real_{+}^{N} \cap \mathrm{lev}_{\le 1}(h_{G, \mathfrak{E}})$ holds by
definition of $\mathcal{R}$ in \eqref{eq.def.regionrate} and
Lemma \ref{lemma.convext.rhoexp}(i),
we have $$\Omega \subset [0, b]^{N} \cap \mathrm{lev}_{\le 1}(h_{G, \mathfrak{E}}) 
\subset \mathcal{R}.$$
By definition of the distance $d_2$ in \eqref{eq.def.EucDist},
$\Omega \subset  \mathcal{R}$ implies $d_2(\signal{r}, \mathcal{R}) \leq d_2(\signal{r}, \Omega)$
for every $\signal{r} \in \real^{N}$, and hence the desired result follows from the claim in (i).

(iv) Since $\Omega$ is a nonempty closed convex set by Proposition \ref{prop.existence.equivalence.prob.rate},
the projection mapping $P_{\Omega}\colon \real^{N}\to\Omega$ is well defined.
Since the solution sets of Problems \eqref{eq.sumrate.rate} and \eqref{eq.refsumrate.frrate}
are the same for sufficiently large $b > 0$ by Proposition \ref{prop.existence.equivalence.prob.rate},
we have $\Omega = \argmax_{\signal{r} \in \mathcal{R}} \bm{w}^{\mathsf{T}} \signal{r}$ for such $b$.
Altogether, the desired result is shown by
\begin{align*}
&\lim_{k \to \infty} \bigl| \bm{w}^{\mathsf{T}}\signal{r}_k-\max_{\signal{r} \in \mathcal{R}} \bm{w}^{\mathsf{T}}\signal{r}\bigr| = \lim_{k \to \infty}| \bm{w}^{\mathsf{T}}\signal{r}_k- \bm{w}^{\mathsf{T}} P_{\Omega}(\signal{r}_k)| \\
 &\overset{\text{(a)}}\leq  \lim_{k \to \infty}\|\signal{w}\|_2\|\signal{r}_k- P_{\Omega}(\signal{r}_k)\|_2
= \|\signal{w}\|_2 \lim_{k \to \infty} d_2(\signal{r}_k, \Omega) \overset{\text{(b)}}= 0,
\end{align*}
where (a) and (b) follow from the Cauchy-Schwarz inequality
and the claim in (i), respectively.
\end{proof}
\begin{remark}
The condition for $(\mu_k)_{k\in\mathbb{N}}$ in Theorem \ref{theo.converge.rate}
is satisfied by, e.g., $(\forall k \in \mathbb{N})~\mu_{k} = ak^{-q}$ for any $a > 0$ and $q \in (0, 1]$.
\end{remark}

\subsection{A low-complexity method for computing subgradients}
\label{sect.concAlg}
To efficiently compute the steps of the proposed algorithm,
we only need an efficient method for finding a subgradient at the first step in \eqref{eq:HSDMmaxsumrate.rate}, since the other steps can be done in $\mathcal{O}(N)$ operations.
Since $(\forall k \in \mathbb{N})~\signal{r}_{k} \in \real_{++}^{N}$
is guaranteed in \eqref{eq:HSDMmaxsumrate.rate} by Theorem \ref{theo.converge.rate}(ii),
it is sufficient to find a subgradient of $h_{G, \mathfrak{E}}$ at a given point in $\mathbb{R}_{++}^{N}$.
To this end, we focus on the PHOC mappings $G$ considered in Example \ref{ex.maxlinearG},
and make the following additional assumption.

\begin{assumption}
\label{asmp.convex.rhoMl.exp}
	In addition to the setting of Example \ref{ex.maxlinearG},
we assume that
$\varrho_{\signal{M}_l} \circ \mathfrak{E}$
is convex on $\real_{++}^{N}$ for every $l \in \{1,\ldots,L\}$.
\end{assumption}
\begin{remark}
\label{remark.inverseZ}
By combining \cite[Theorem~4.3]{friedland81} and \cite[Proposition 6]{renatoconvex},
we deduce that Assumption \ref{asmp.convex.rhoMl.exp} holds if $\signal{M}_l$ is an \emph{inverse Z-matrix} \cite{nussbaum1986convexity,johnson2011} (i.e., $\signal{M}_l$ has an inverse matrix whose off-diagonal components are nonpositive) for every $l \in \{1,\ldots,L\}$.
However, this condition is just sufficient for Assumption \ref{asmp.convex.rhoMl.exp}:
$\varrho_{\signal{M}_l} \circ \mathfrak{E}$ can be convex even if  $\signal{M}_l$ is not an inverse Z-matrix.
For validity of Assumption \ref{asmp.convex.rhoMl.exp} in wireless applications, see Sect.~\ref{sect.wireless}.
\end{remark}

\begin{lemma}
\label{lemm.conv.maxlinG}
Consider the setting of Example \ref{ex.maxlinearG}. If Assumption \ref{asmp.convex.rhoMl.exp} holds,
then $G$ given in Example \ref{ex.maxlinearG} satisfies Assumption \ref{asmp.convex.extension.rho.exp}.
\end{lemma}
\begin{proof}
Assumption \ref{asmp.G.hypTnorm}, which is required by Assumption \ref{asmp.convex.extension.rho.exp}, holds by definition of $G$ in Example \ref{ex.maxlinearG}.
By definition of $\mathfrak{E}$ in \eqref{eq.def.VecExp},
we have $\mathfrak{E}(\signal{x}) \in \real_{+}^{N}$ for every $\signal{x} \in \real_{+}^{N}$.
Hence, by the equality in \eqref{eq.rho.maxlin}, we have
\begin{align}
\label{eq.max.rhoexp}
(\forall \signal{x} \in \real_{+}^{N})\quad (\spradG \circ \mathfrak{E})(\signal{x})
= \max_{l\in\{1,\ldots,L\}} (\varrho_{\signal{M}_l} \circ \mathfrak{E})(\signal{x}).
\end{align}
This shows that $\spradG \circ \mathfrak{E}$ is the pointwise maximum of
 convex functions on $\real_{++}^{N}$, and hence $\spradG \circ \mathfrak{E}$ is convex on $\real_{++}^{N}$.
\end{proof}

Under Assumption \ref{asmp.convex.rhoMl.exp}, the following proposition provides an efficient way to compute a subgradient of $h_{G, \mathfrak{E}}$ at any point in the positive orthant.

\begin{proposition}
\label{prop.subgrad.rate.maxlinG}
Consider the setting of Example \ref{ex.maxlinearG}.
Fix $\signal{r} \in \real_{++}^{N}$ arbitrarily, and
let
$$l^{\star} 
\in \argmax_{l \in \{1,\ldots,L\}} (\varrho_{\signal{M}_l} \circ \mathfrak{E})(\signal{r})
= \argmax_{l \in \{1,\ldots,L\}}\rho(\mathrm{diag}(\mathfrak{E}(\signal{r})) \signal{M}_l).$$
Then, under Assumption \ref{asmp.convex.rhoMl.exp}, we have
\begin{align}
\label{eq:ex.subgrad.rate}
\frac{\mathrm{diag}((e^{r_n})_{n \in \mathcal{N}})\mathrm{diag}(\bm{\eta})\signal{M}_{l^\star}\bm{\xi}}{\bm{\eta}^{\mathsf{T}} \bm{\xi}} \in \partial h_{G, \mathfrak{E}}(\signal{r}),
\end{align}
where $\signal{\xi} \in \real_{++}^{N}$ and $\signal{\zeta} \in \real_{++}^{N}$
are positive right and left eigenvectors of $\mathrm{diag}(\mathfrak{E}(\signal{r}))\signal{M}_{l^\star}\in\real_{++}^{N \times N}$
corresponding to the spectral radius $\rho(\mathrm{diag}(\mathfrak{E}(\signal{r}))\signal{M}_{l^\star}) > 0$, i.e., positive vectors satisfying
$\mathrm{diag}(\mathfrak{E}(\signal{r}))\signal{M}_{l^\star}\signal{\xi} = \rho(\mathrm{diag}(\mathfrak{E}(\signal{r}))\signal{M}_{l^\star}) \signal{\xi}$
and $\signal{\eta}^{\mathsf{T}}\mathrm{diag}(\mathfrak{E}(\signal{r}))\signal{M}_{l^\star} = \rho(\mathrm{diag}(\mathfrak{E}(\signal{r}))\signal{M}_{l^\star}) \signal{\eta}^{\mathsf{T}}$, respectively.\footnote{\label{footnote.lreigenvec}%
Since $\mathrm{diag}(\mathfrak{E}(\signal{r}))\signal{M}_{l^\star}$ is a positive matrix,
$\rho(\mathrm{diag}(\mathfrak{E}(\signal{r}))\signal{M}_{l^\star})$ is identical to the spectral radius of $\mathrm{diag}(\mathfrak{E}(\signal{r}))\signal{M}_{l^\star}$ in standard linear algebra, and hence the existence of positive right and left eigenvectors is guaranteed by standard Perron-Frobenius theory (see, e.g., \cite[Theorem 8.2.8]{horn12}).}

\end{proposition}
\begin{proof}
Since Lemma \ref{lemm.conv.maxlinG} shows that Assumption \ref{asmp.convex.rhoMl.exp} implies Assumption \ref{asmp.convex.extension.rho.exp},
we can construct $h_{G, \mathfrak{E}}$ defined in Lemma \ref{lemma.convext.rhoexp}.
Meanwhile, for every $l \in \{1,\ldots,L\}$,
$G_l\colon \real_+^N\to\real_{+}^N\colon \signal{x}\mapsto \signal{M}_l\signal{x}$
satisfies Assumption \ref{asmp.G.hypTnorm} by Fact \ref{fact.linearG}.
Owing to Assumption \ref{asmp.convex.rhoMl.exp},
$G_l$ satisfies the conditions in Assumption \ref{asmp.convex.extension.rho.exp}.
Hence, for every $l \in \{1,\ldots,L\}$,
applying Lemma \ref{lemma.convext.rhoexp}, we deduce that the function
$$
h_{G_l, \mathfrak{E}} \colon \real^{N}\to\real_{+}\colon \bm{x} \mapsto (\varrho_{\signal{M}_l} \circ \mathfrak{E} \circ P_{\real_{+}^{N}})(\bm{x})
$$
is convex on $\real^{N}$, and
\begin{align}
\label{eq.conext.rhoMl}
(\forall \signal{x} \in \real_{+}^{N})\quad h_{G_l, \mathfrak{E}}(\signal{x}) = (\varrho_{\signal{M}_l} \circ \mathfrak{E})(\signal{x}).
\end{align}
Since $(\forall \signal{x} \in \real^{N})~P_{\real_{+}^{N}}(\signal{x}) \in \real_{+}^{N}$ holds,
using the equality in \eqref{eq.max.rhoexp},
we further obtain
\begin{align}
\label{eq.fGEmax}
(\forall \signal{x} \in \real^{N}) \quad h_{G, \mathfrak{E}}(\signal{x}) = \max_{l \in \{1,\ldots,L\}} h_{G_l, \mathfrak{E}} (\signal{x}).
\end{align}
Now consider $\signal{r}$ and $l^{\star}$ given in the statement of this proposition.
Since $h_{G_l, \mathfrak{E}}$ is a convex function from $\real^{N}$ to $\real$ for every $l \in \{1,\ldots,L\}$,
$\signal{g} \in \partial h_{G_{l^\star}, \mathfrak{E}}(\signal{r})$
implies
$\signal{g} \in \partial h_{G, \mathfrak{E}}(\signal{r})$
by \eqref{eq.fGEmax} and Fact \ref{fact.subdifmaxconv} in Appendix \ref{appendix.proof.subgradient}.
Hence it is sufficient to show that the left-hand side of \eqref{eq:ex.subgrad.rate}
is a subgradient of $h_{G_{l^\star}, \mathfrak{E}}$ at $\signal{r}$.
By definition of $\mathfrak{E}$ in \eqref{eq.def.VecExp},
$\mathfrak{E}$ is Fr\'{e}chet differentiable at $\bm{r} \in \real_{++}^{N}$, and
$\mathfrak{E}(\signal{r}) \in \real_{++}^{N}$ holds.
Since $\signal{M}_{l^\star} \in \real_{++}^{N \times N}$ holds in the setting of Example \ref{ex.maxlinearG},
$\varrho_{\signal{M}_{l^\star}}$ is Fr\'{e}chet differentiable at $\mathfrak{E}(\signal{r}) \in \mathbb{R}_{++}^{N}$ by Fact \ref{fact.diff.rhodiagposM}.
Hence, by the chain rule (see, e.g.,~\cite[Fact 2.63]{baus17}),
$\varrho_{\signal{M}_{l^\star}} \circ \mathfrak{E}$
is (Fr\'{e}chet) differentiable at $\signal{r} \in \mathbb{R}_{++}^{N}$, and its gradient is given by
\begin{align*}\nabla (\varrho_{\signal{M}_{l^\star}} \circ \mathfrak{E})(\signal{r})
&= \mathrm{diag}((e^{r_n})_{n \in \mathcal{N}}) \nabla \varrho_{\signal{M}_{l^\star}} (\mathfrak{E}(\signal{r}))\\
&= \mbox{the left-hand side of \eqref{eq:ex.subgrad.rate}},
\end{align*}
where the last equality follows from the expression of $\nabla \varrho_{\signal{M}_{l^\star}}$ in Fact \ref{fact.diff.rhodiagposM}.
Since the equality in \eqref{eq.conext.rhoMl}
holds in a neighborhood of $\signal{r} \in \real_{++}^{N}$,
$\nabla (\varrho_{\signal{M}_{l^\star}} \circ \mathfrak{E})(\signal{r})$ is
the gradient of $h_{G_{l^\star}, \mathfrak{E}}$ at $\signal{r}$, and hence is a (unique) subgradient of $h_{G_{l^\star}, \mathfrak{E}}$ at $\signal{r}$ by \cite[Proposition~17.31]{baus17}.
\end{proof}

\begin{algorithm}[t]
\SetCommentSty{textrm}
%\DontPrintSemicolon
\setstretch{1.2}
\small
%\LinesNumbered
\caption{for Problem \eqref{eq.sumrate.rate} using Example \ref{ex.maxlinearG}}
\label{alg.rate.Gmaxlin}
\KwIn{$\signal{M}_l \in\real_{++}^{N\times N}~(l=1,\ldots,L)$, $\bm{w} \in \real_{++}^N$, $\signal{r}_1 \in \mathbb{R}_{++}^{N}$, $(\mu_k)_{k\in\mathbb{N}} \subset \real_{++}^{N}$
satisfying $\lim_{k\to\infty}\mu_k = 0$ and $\sum_{k\in\mathbb{N}}\mu_k = \infty$, and large enough $b > 0$.}
\For{$k = 1,2,\ldots$}{
	$\gamma_k = \max_{l \in \{1,\ldots,L\}}\rho(\mathrm{diag}(\mathfrak{E}(\signal{r}_k))\signal{M}_l)$\;
	Choose $l_{k}^{\star} \in \argmax_{l \in \{1,\ldots,L\}}\rho(\mathrm{diag}(\mathfrak{E}(\signal{r}_k))\signal{M}_l)$\;
	\If{$\gamma_k > 1$}{
		Find $\signal{\xi}_k \in \real_{++}^{N}$ s.t.
		$\mathrm{diag}(\mathfrak{E}(\signal{r}_k))\signal{M}_{l_{k}^{\star}}\signal{\xi}_k = 
		\gamma_k\signal{\xi}_k$\;
		Find $\signal{\eta}_k \in \real_{++}^{N}$ s.t.
		$\signal{\eta}_k^{\mathsf{T}}\mathrm{diag}(\mathfrak{E}(\signal{r}_k))\signal{M}_{l_{k}^{\star}} = 
		\gamma_k\signal{\eta}_k^{\mathsf{T}}$\;
		$\signal{g}_k =  (\mathrm{diag}((e^{r_n})_{n \in \mathcal{N}})\mathrm{diag}(\signal{\eta}_k)\signal{M}_{l_{k}^{\star}}\signal{\xi}_k)/(\signal{\eta}_k^{\mathsf{T}}\signal{\xi}_k)$\;
		$\hat{\signal{r}}_{k} = \signal{r}_k - (\gamma_k-1)\signal{g}_k/\|\signal{g}_k\|_2^2$\;
	}\Else{$\hat{\signal{r}}_{k} = \signal{r}_k$\;}
	$\tilde{\signal{r}}_{k} = P_{[0,b]^{N}}(\hat{\signal{r}}_{k})$\tcp*{See \eqref{eq.def.projectionbox}}
	$\signal{r}_{k+1} = \tilde{\signal{r}}_{k} - \mu_{k}\nabla\Psi(\tilde{\signal{r}}_{k})$\tcp*{See Lemma \ref{lemma.LipGradCost.rate}}
}
\end{algorithm}

\begin{remark}[Optimality of Algorithm \ref{alg.rate.Gmaxlin}]
\label{remark.optimality.rate}
Algorithm \ref{alg.rate.Gmaxlin} is a particular instance of the proposed algorithm \eqref{eq:HSDMmaxsumrate.rate}, obtained by employing
the subgradient characterized in Proposition \ref{prop.subgrad.rate.maxlinG}, which is valid
under Assumption \ref{asmp.convex.rhoMl.exp}.
Since Assumption \ref{asmp.convex.rhoMl.exp} implies
that $G$ given in Example \ref{ex.maxlinearG} satisfies
Assumption \ref{asmp.convex.extension.rho.exp} (see Lemma \ref{lemm.conv.maxlinG}),
convergence of
Algorithm \ref{alg.rate.Gmaxlin}
to the global optimal value of Problem \eqref{eq.sumrate.rate} -- and hence that of the original problem \eqref{eq.sumrate.power} -- is guaranteed under Assumption \ref{asmp.convex.rhoMl.exp} by Theorem \ref{theo.converge.rate}.

Meanwhile, convergence of Algorithm \ref{alg.sinr.Gmaxlin} -- developed in the Supplemental Material for Problem \eqref{eq.sumrate.sinr} -- to the global optimal value of Problem \eqref{eq.sumrate.power} is proved under the convexity of
$\varrho_{\signal{M}_l}\colon\real_{+}^{N}\to\real_{+}$ for every $l \in \{1,\ldots,L\}$.
This convexity implies Assumption \ref{asmp.convex.rhoMl.exp} by \cite[Proposition 6]{renatoconvex},
but the converse does not hold in general (see Sect.~\ref{sect.wireless.simu} for examples). Thus, convergence of Algorithm \ref{alg.rate.Gmaxlin}
to the global optimal value of Problem \eqref{eq.sumrate.power} is guaranteed in substantially more general scenarios than  Algorithm \ref{alg.sinr.Gmaxlin}.
\end{remark}

\begin{remark}[Computational cost of Algorithm \ref{alg.rate.Gmaxlin}]
At each iteration $k \in \mathbb{N}$, the computational cost of Algorithm \ref{alg.rate.Gmaxlin} is dominated by the evaluation of $\rho(\mathrm{diag}(\mathfrak{E}(\signal{r}_k))\signal{M}_l)$
for each $l = 1,\ldots,L$.
Since $\mathrm{diag}(\mathfrak{E}(\signal{r}_k))\signal{M}_l$
is guaranteed to be a positive matrix by
Theorem \ref{theo.converge.rate}(ii)
and the setting of Example \ref{ex.maxlinearG},
we just need to compute the spectral radius in the sense of standard linear algebra,
i.e., the largest absolute eigenvalue of the positive matrix
$\mathrm{diag}(\mathfrak{E}(\signal{r}_k))\signal{M}_l$ for each $l = 1,\ldots,L$.
Hence, the computational cost of Algorithm \ref{alg.rate.Gmaxlin} is considered fairly low
because we can exploit the rich literature on linear algebra (e.g., \cite[Ch.~7]{golub})
for efficient computation of largest absolute eigenvalues of positive matrices.
\end{remark}

\begin{remark}[Applicability of Algorithm \ref{alg.rate.Gmaxlin} to general scenarios in Example \ref{ex.maxlinearG}]
Although Algorithm \ref{alg.rate.Gmaxlin} is derived under Assumption \ref{asmp.convex.rhoMl.exp},
even without this assumption,
the steps of Algorithm \ref{alg.rate.Gmaxlin} are computable
if $\signal{M}_l \in \mathbb{R}_{++}^{N \times N}$ for each $l \in \{1,\ldots,L\}$.
At the first iteration $k = 1$, since $\signal{r}_1 \in \real_{++}^{N}$ by assumption,
there exist positive right and left eigenvectors of $\mathrm{diag}(\mathfrak{E}(\signal{r}_1))\signal{M}_l \in \mathbb{R}_{++}^{N \times N}$ corresponding to the spectral radius $\rho(\mathrm{diag}(\mathfrak{E}(\signal{r}_1))\signal{M}_l) > 0$ (see Footnote \ref{footnote.lreigenvec}). Hence, Algorithm \ref{alg.rate.Gmaxlin} can compute $\hat{\signal{r}}_{1}$ and then $\tilde{\signal{r}}_1$ and $\signal{r}_2$.
Similarly to the proof of Theorem \ref{theo.converge.rate}(ii), we can ensure
$\signal{r}_2 \in \real_{++}^{N}$, and hence $\hat{\signal{r}}_{2}$ is computable.
Repeating this process, we deduce that the whole sequence $(\bm{r}_k)_{k \in \mathbb{N}}$ can be generated by Algorithm \ref{alg.rate.Gmaxlin} even if Assumption \ref{asmp.convex.rhoMl.exp} does not hold. This fact is important in practice because it shows that there is no need to verify Assumption \ref{asmp.convex.rhoMl.exp} to apply Algorithm \ref{alg.rate.Gmaxlin} to, e.g., resource allocation in cell-less networks (see also Sect.~\ref{sect.wireless} below).
\end{remark}

\section{Applications in cell-less networks}
\label{sect.wireless}
\subsection{System model}
\label{sect.wireless.system}
To illustrate the preceding results in a concrete application, we consider the uplink of a cell-less network model similar to \cite{miretti2022closed,renatoconvex}. Specifically, $N\in\Natural$ users represented by the set $\mathcal{N}=\{1,\ldots,N\}$ transmit data to $K_1\in\Natural$ access points, each equipped with $K_2\in\Natural$ antennas. For each user $n\in\mathcal{N}$, the aggregated channel to all access points is modeled by a random vector $\signal{h}_n$ with samples taking values in $\mathbb{C}^{K_1K_2}$. (For brevity, the underlying probability space is omitted, and all subsequent operations on random variables, such as expectations, are assumed to be well defined.) Likewise, the joint beamforming vector applied by all access point to the signal of user $n\in\mathcal{N}$ is modeled as a $\mathbb{C}^{K_1K_2}$-valued random vector denoted by $\signal{v}_n$. We assume that the beamformer $\signal{v}_n$ is not a function of the transmit power vector $\signal{p}= [p_1,\ldots,p_N]^{\mathsf{T}} \in \real_+^N$, where $p_n$ denotes the transmit power of user $n\in\mathcal{N}$. A classical beamforming scheme satisfying this assumption, which is also used in our numerical simulations, is conjugate beamforming. In this scheme, the beamformers $(\signal{v}_n)_{n\in\mathcal{N}}$ are chosen as the corresponding channel estimates $(\signal{h}_n)_{n\in\mathcal{N}}$, with entries set to zero at the coordinates of access points that do not participate in detecting the symbols of user $n$. In this setting, given a power vector $\signal{p}\in\real_{++}^N$, the achievable rate $r_n$ of user $n\in\mathcal{N}$ based on the UatF bound is given by $r_n:=\log(1+s_n)$ [nats per symbol], where $s_n$ is the effective SINR \cite{miretti2022closed}:
	\begin{align}
		\label{eq.sinrcl}
		s_n:=\dfrac{p_n |E[\signal{h}_n^\mathsf{H}\signal{v}_n]|^2}{p_n V(\signal{h}^\mathsf{H}_n\signal{v}_n)+\sum \limits_{l \in\mathcal{N}\backslash\{n\}}p_l E[|\signal{h}_l^\mathsf{H}\signal{v}_n|^2]+E[\|\signal{v}_n\|_2^2]},
	\end{align}
	and  $E(\cdot)$ and $V(\cdot)$ denote, respectively, the expectation and the variance of a random variable.
With the above definitions, weighted sum-rate maximization based on the UatF bound can be expressed as in Problem \eqref{eq.sumrate.power} by setting
\begin{align}
\label{eq.UatFint}
f_n\colon \real_+^N\to\real_{++}\colon\bm{p}\mapsto\bm{m}_n^{\mathsf{T}}\bm{p} + u_n
\end{align}
for every $n \in \mathcal{N}$, where $\bm{m}_n := |E[\signal{h}_n^\mathsf{H}\signal{v}_n]|^{-2}\bm{c}_n \in \real_{++}^{N}$;
$\bm{c}_n\in \real_{++}^{N}$ is the vector with $l$th entry given by $V(\signal{h}^\mathsf{H}_n\signal{v}_n)$ if $l=n$, or $E[|\signal{h}_l^\mathsf{H}\signal{v}_n|^2]$ otherwise; and $u_n := |E[\signal{h}_n^\mathsf{H}\signal{v}_n]|^{-2}E[\|\signal{v}_n\|_2^2] > 0$. In particular, for conjugate beamforming, the vectors
$\bm{m}_n ~ (n \in \mathcal{N})$ and $\signal{u}$ are typically positive, which is the assumption we adopt here.

If we impose a per-user power constraint $p_\text{max}>0$, the set of feasible power is given by $\mathcal{P}=\{\signal{p}\in\real_{+}^N \mid \|\signal{p}\|\le 1 \}$, where $\|\cdot\|$ is the scaled $\ell_\infty$-norm defined by
\begin{multline*}
(\forall\signal{p}\in\real^N)\quad \|\signal{p}\| := (1/p_\text{max})~\|\signal{p}\|_\infty=\max_{l\in\{1,\ldots,N\}}\signal{a}_l^{\mathsf{T}}|\signal{p}|,
\end{multline*}
and $\signal{a}_l\in\real_+^N$ is the vector of zeros, expect for its $l$th entry, which is set to $1/p_\text{max}$. In this setting, as shown in \cite{renatoconvex} and reviewed in Sect.~\ref{sect.intro.reformulation}, Problem~\refeq{eq.sumrate.power} can be written using achievable rates as the optimization variables as Problem~\refeq{eq.sumrate.rate} by using the following mapping $G$ in the definition of the achievable rate region $\mathcal{R}$ in \refeq{eq.def.regionrate}:
\begin{align*}
G:\real_+^N\to\real_{+}^N:\signal{p}\mapsto \signal{Mp}+\signal{u}\|\signal{p}\|,
\end{align*}
where $\signal{M} := [\signal{m}_1,\ldots,\signal{m}_N]^\mathsf{T}$ and
$\signal{u}:=[u_1,\ldots,u_N]^\mathsf{T}$. The matrix $\signal{M}$ and the vector $\signal{u}$ are typically positive for conjugate beamforming, implying that the setting of Example~\ref{ex.maxlinearG} is verified, so $\varrho_G$ in Problem~\refeq{eq.sumrate.rate} can be expressed as in \refeq{eq.rho.maxlin} with $L=N$. As a result, Algorithm~\ref{alg.rate.Gmaxlin} provably converges to the global optimal value of Problem~\refeq{eq.sumrate.rate} whenever Assumption \ref{asmp.convex.rhoMl.exp} holds (see Remark \ref{remark.optimality.rate}). We recall that a sufficient (though not necessary) condition for Assumption \ref{asmp.convex.rhoMl.exp} is that the matrices $(\signal{M}_l)_{l\in\{1,\ldots,N\}}$ introduced in Example~\ref{ex.maxlinearG} are inverse Z-matrices.

\subsection{Simulations}
\label{sect.wireless.simu}

For the numerical simulations, we consider the same settings as in \cite[Sect.~V]{renatoconvex},\cite{miretti2022closed}. Briefly, we consider the uplink of a small cell-less network with four access points, each equipped with two antennas, placed uniformly at random in a $100\mathrm{m} \times 100\mathrm{m}$ square area. There are three single-antenna users also dropped uniformly at random in the same area, and each user is associated with the two access points that provide the strongest channels. All other simulations parameters (transmit frequency, system bandwidth, channel models, etc.) are the same as described in \cite[Sect.~III-D]{miretti2022closed}, so we omit the details for brevity.

In the simulations, $f_n~(n \in \mathcal{N})$ in Problem \eqref{eq.sumrate.power} is defined according to \eqref{eq.UatFint} using the quantities in the UatF bound shown in \eqref{eq.sinrcl}, where the expectations are approximated via empirical averages with 100 samples of the channels. 
For simplicity, we use the uniform weight $(\forall n \in \mathcal{N})~w_n = 1$ in Problems \eqref{eq.sumrate.power}, \eqref{eq.sumrate.rate}, and  \eqref{eq.sumrate.sinr}.

We first investigate the validity of Assumption \ref{asmp.convex.rhoMl.exp} in $1000$ simulations of user drops.
The matrices $(\signal{M}_l)_{l\in\{1,\ldots,N\}}$ have been found to be inverse Z-matrices in $19.3$\% of simulations. Since this condition is just sufficient for Assumption \ref{asmp.convex.rhoMl.exp}, we also check the necessary condition for convexity of $\varrho_{\bm{M}_l}\circ \mathfrak{E}\colon\real_{+}^{N}\to\real_{+}~(l =1,\ldots,N)$ on discrete points:
\begin{multline*}
(\varrho_{\bm{M}_l} \circ \mathfrak{E}) (\alpha \bar{\bm{r}}_1 + (1-\alpha)\bar{\bm{r}}_2)\\
\leq \alpha (\varrho_{\bm{M}_l} \circ \mathfrak{E})(\bar{\bm{r}}_1) +  (1-\alpha)(\varrho_{\bm{M}_l} \circ \mathfrak{E})(\bar{\bm{r}}_2),
\end{multline*}
where $10000$ pairs of $\bar{\bm{r}}_1$ and $\bar{\bm{r}}_2$ are uniformly drawn from $[0, 5]$, and $\alpha$ is selected from $99$ uniformly spaced points between $0.01$ and $0.99$. This necessary condition has been satisfied in $99.7$\% of simulations, suggesting that Assumption \ref{asmp.convex.rhoMl.exp} may hold true in most simulations. Similarly, we check the necessary condition for convexity of $\varrho_{\bm{M}_l}\colon\real_{+}^{N}\to\real_{+}~(l =1,\ldots,N)$
-- assumed for Algorithm \ref{alg.sinr.Gmaxlin} developed in the Supplemental Material for Problem \eqref{eq.sumrate.sinr} -- on discrete points:
$$
\varrho_{\bm{M}_l} (\alpha \bar{\bm{s}}_1 + (1-\alpha)\bar{\bm{s}}_2)
\leq \alpha \varrho_{\bm{M}_l}(\bar{\bm{s}}_1) +  (1-\alpha)\varrho_{\bm{M}_l}(\bar{\bm{s}}_2) + \epsilon,
$$
where $\epsilon = 10^{-13}$ is introduced to avoid the issue of numerical error. This condition has been satisfied in $77.7$\% of simulations.
Hence, Assumption \ref{asmp.convex.rhoMl.exp} is expected to hold in substantially more general scenarios than the convexity condition of $\varrho_{\bm{M}_l}~(l =1,\ldots,N)$.
Note that we have the mathematical fact that convexity of $\varrho_{\bm{M}_l}$
implies convexity of $\varrho_{\bm{M}_l}\circ\mathfrak{E}$ \cite[Proposition 6]{renatoconvex}.

Next, we examine the convergence of Algorithm \ref{alg.rate.Gmaxlin} for Problem \eqref{eq.sumrate.rate} and Algorithm \ref{alg.sinr.Gmaxlin} for Problem \eqref{eq.sumrate.sinr}. As a benchmark, we also consider the WMMSE method from \cite{shi2011iteratively}, applied to UatF-based instances of Problem~\eqref{eq.sumrate.power}. The WMMSE method is a well-known heuristic widely used in the cell-less literature \cite{demir2021,miretti2025two}. It is a block coordinate descent algorithm applied to a weighted MMSE reformulation of the sum-rate optimization problem. Since this reformulation is typically nonconvex in general, the behavior of the WMMSE algorithm is strongly initialization-dependent: different starting points may lead to different suboptimal points, even in cases where Problems \eqref{eq.sumrate.rate} and \eqref{eq.sumrate.sinr} are convex. For example, following the recommendation in \cite[Alg.~7.2]{demir2021} to initialize within the feasible set $\mathcal{P}$, we verify that the WMMSE algorithm tends to converge to different (typically suboptimal) points when started from (a) $\bm{p}_1 = p_\text{max}[1,1,1]^{\mathsf{T}} \in \mathcal{P}$ and (b) $\bm{p}_1 = p_\text{max}[1,0,1]^{\mathsf{T}}\in \mathcal{P}$.\footnote{%
The authors gratefully acknowledge Isabel von Stebut for the suggestion on initialization of the WMMSE method.}

Fig.~\ref{fig.simu.invZ} shows the results of a user placement scenario where $(\signal{M}_l)_{l\in\{1,\ldots,N\}}$ are inverse Z-matrices.
In this scenario, convergence of Algorithm \ref{alg.rate.Gmaxlin} to the global optimum of Problem \eqref{eq.sumrate.rate} is ensured by Theorem \ref{theo.converge.rate} if the initial point $\bm{r}_1$ is positive and the positive step-size $(\mu_k)_{k\in\mathbb{N}}$ satisfies $\lim_{k\to\infty}\mu_k = 0$ and $\sum_{k\in\mathbb{N}}\mu_k = \infty$.
Likewise, Problem \eqref{eq.sumrate.sinr} is also guaranteed to be convex in this setting, and convergence of Algorithm \ref{alg.sinr.Gmaxlin} to the global optimum of Problem \eqref{eq.sumrate.sinr} is also guaranteed by Theorem \ref{theo.converge.sinr} in the Supplemental Material if the initial point $\bm{s}_1$ is positive and the positive step-size $(\mu_k)_{k\in\mathbb{N}} $ satisfies the above condition.
In particular, we use $\bm{r}_1 = [0.5, 0.5, 0.5]^{\mathsf{T}}$ and $(\forall k \in \mathbb{N})~\mu_k = 0.4k^{-0.999}$ for Algorithm \ref{alg.rate.Gmaxlin}, and $\bm{s}_1 = [0.5, 0.5, 0.5]^{\mathsf{T}}$
and $(\forall k \in \mathbb{N})~\mu_k = 1.6k^{-0.999}$ for Algorithm \ref{alg.sinr.Gmaxlin}.

\begin{figure}[t]
\centering
\subfloat[]{\includegraphics[width=0.5\columnwidth]{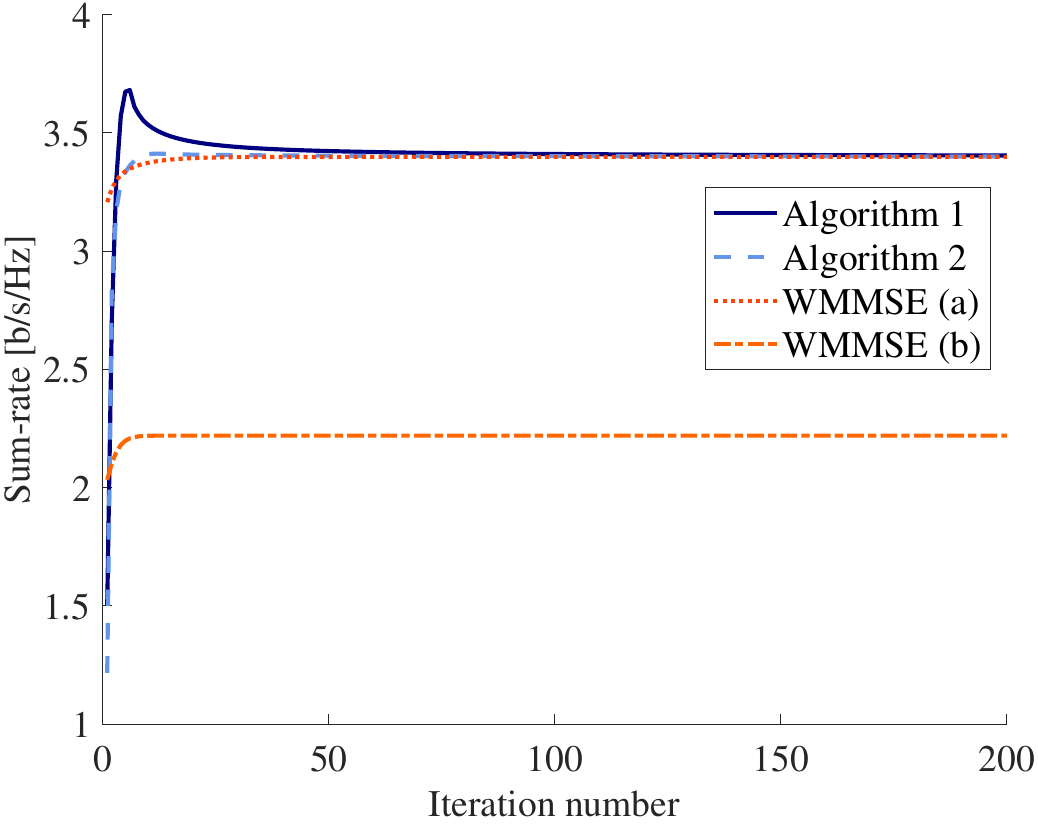}}
\hfill
\centering
\subfloat[]{\includegraphics[width=0.5\columnwidth]{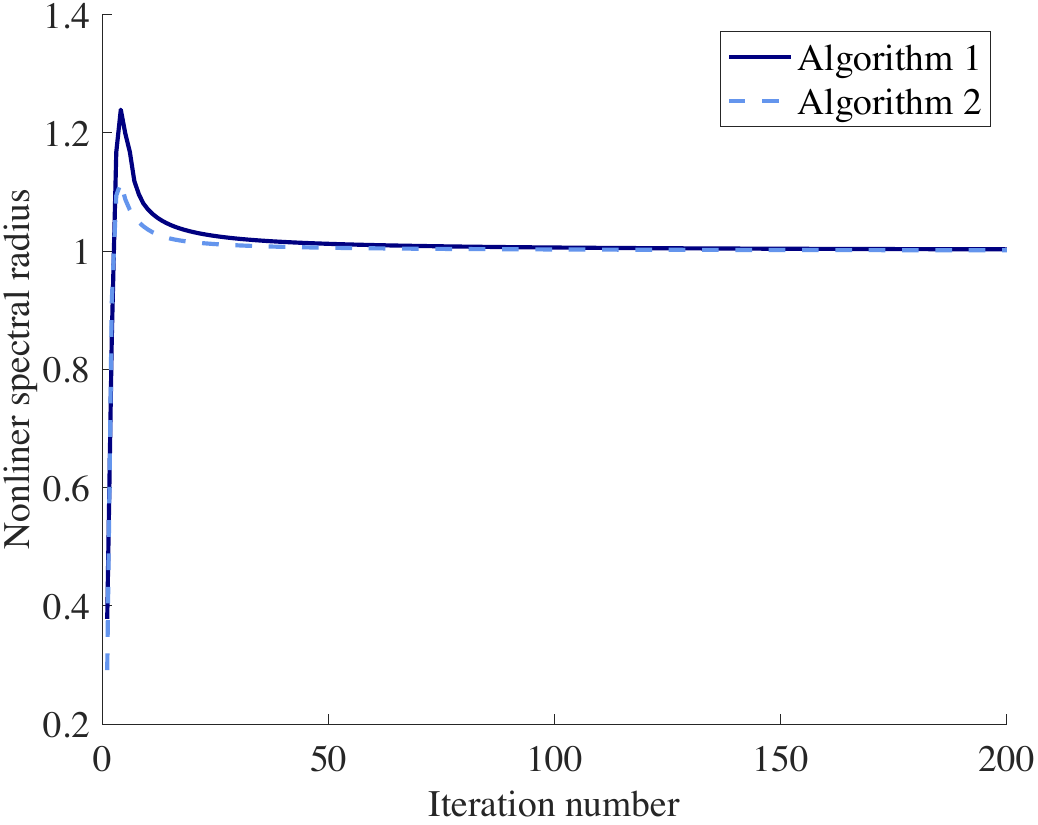}}
\caption{Behavior of algorithms for a simulation where $\bm{M}_l$ is an inverse Z-matrix for every $l \in \{1,\ldots,L\}$.}
\label{fig.simu.invZ}
\end{figure}

Fig.~\ref{fig.simu.invZ}(a) depicts the sum-rate at iteration $k \in \{1,\ldots,200\}$: 
$\bm{w}^\mathsf{T}\bm{r}_k$ with $\bm{r}_k$ generated by Algorithm \ref{alg.rate.Gmaxlin},
$\sum_{n \in \mathcal{N}}w_n\log(1 + s_{k,n})$ with $[s_{k,1},\ldots,s_{k,N}]^\mathsf{T} = \signal{s}_k$ generated by Algorithm \ref{alg.sinr.Gmaxlin},
and $\sum_{n \in \mathcal{N}}w_n\log(1 + p_{k,n}/f_{n}(\signal{p}_k))$ with $ = [p_{k,1},\ldots,p_{k,N}]^\mathsf{T} = \signal{p}_k$ generated by the WMMSE method using the initial points (a) and (b), respectively.
Fig.~\ref{fig.simu.invZ}(b) depicts the value of the nonlinear spectral radius function at iteration $k \in \{1,\ldots,200\}$: $(\spradG \circ \mathfrak{E}) (\bm{r}_k)$ for Algorithm \ref{alg.rate.Gmaxlin}
and $\spradG(\bm{s}_k)$ for Algorithm \ref{alg.sinr.Gmaxlin}.
Our theoretical results are validated in Fig.~\ref{fig.simu.invZ}(a) and (b) because Algorithms \ref{alg.rate.Gmaxlin} and \ref{alg.sinr.Gmaxlin} both converge to the same sum-rate value, and they eventually satisfy the constraints in Problems \eqref{eq.sumrate.rate} and \eqref{eq.sumrate.sinr}.
Note that the nonnegative constraints in Problems \eqref{eq.sumrate.rate} and \eqref{eq.sumrate.sinr} are satisfied at every iteration $k \in \mathbb{N}$ (see Theorems \ref{theo.converge.rate}(ii) and \ref{theo.converge.sinr}(ii)).

In Fig.~\ref{fig.simu.invZ}, Algorithms~\ref{alg.rate.Gmaxlin} and \ref{alg.sinr.Gmaxlin} exhibit convergence speed comparable to the fast WMMSE heuristic. Using the solutions produced by Algorithms~\ref{alg.rate.Gmaxlin} and \ref{alg.sinr.Gmaxlin} as globally optimal references, we observe that the WMMSE algorithm initialized at point (a) also appears to converge to the global optimal value in this experiment. However, unlike the proposed methods, it seems difficult -- beyond the specific setting considered here -- to characterize which starting points for the WMMSE algorithm lead to the global optimal value.

In Fig.~\ref{fig.simu.notinvZ}, we use the same settings used to produce \cite[Fig.~3]{renatoconvex}. It corresponds to a user placement scenario where $\signal{M}_{l}$ is not an inverse Z-matrix for every $l \in \{1,\ldots,N\}$. While the condition on $\signal{M}_{l}$ to be an inverse Z-matrix is just sufficient for convexity of $\varrho_{\bm{M}_l}$, we have numerically confirmed that $\varrho_{\bm{M}_l}$ is not convex for every $l \in \{1,\ldots,N\}$. On the other hand, numerical evaluations of $\varrho_{\bm{M}_l} \circ \mathfrak{E}~(l =1,\ldots,N)$ suggest that
Assumption \ref{asmp.convex.rhoMl.exp} may indeed hold. Hence, we conjecture that the convergence guarantee in Theorem~\ref{theo.converge.rate} also applies in this scenario. The initializations and step-sizes for Algorithms~\ref{alg.rate.Gmaxlin} and \ref{alg.sinr.Gmaxlin} are chosen exactly as in the experiment of Fig.~\ref{fig.simu.invZ}. While the guarantee in Theorem~\ref{theo.converge.sinr} for Algorithm~\ref{alg.sinr.Gmaxlin} no longer applies here since $\varrho_{\bm{M}_l}$ is not convex, both Algorithms~\ref{alg.rate.Gmaxlin} and \ref{alg.sinr.Gmaxlin} empirically converge to the same sum-rate value in this setting. Nevertheless, we recommend using Algorithm~\ref{alg.rate.Gmaxlin} in general, because its convergence to the global optimum of Problem~\eqref{eq.sumrate.rate} is guaranteed under substantially more general conditions than those available for Algorithm~\ref{alg.sinr.Gmaxlin}.

\begin{figure}[t]
\centering
\subfloat[]{\includegraphics[width=0.5\columnwidth]{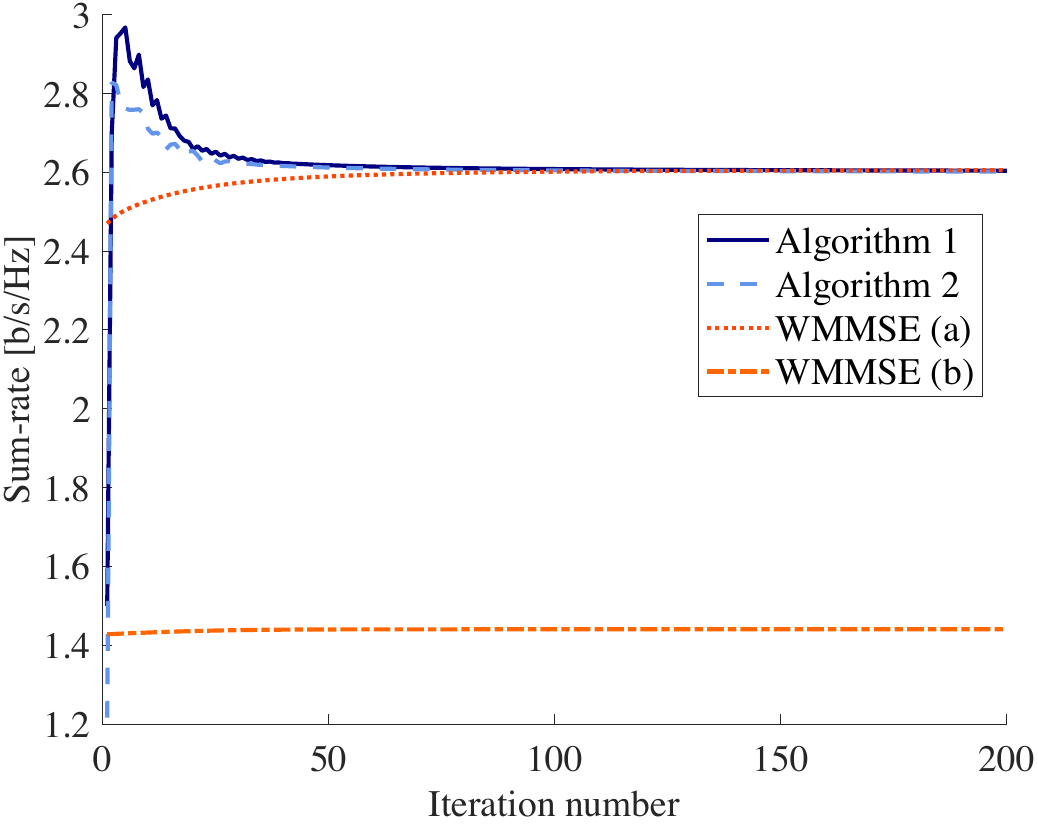}}
\hfill
\centering
\subfloat[]{\includegraphics[width=0.5\columnwidth]{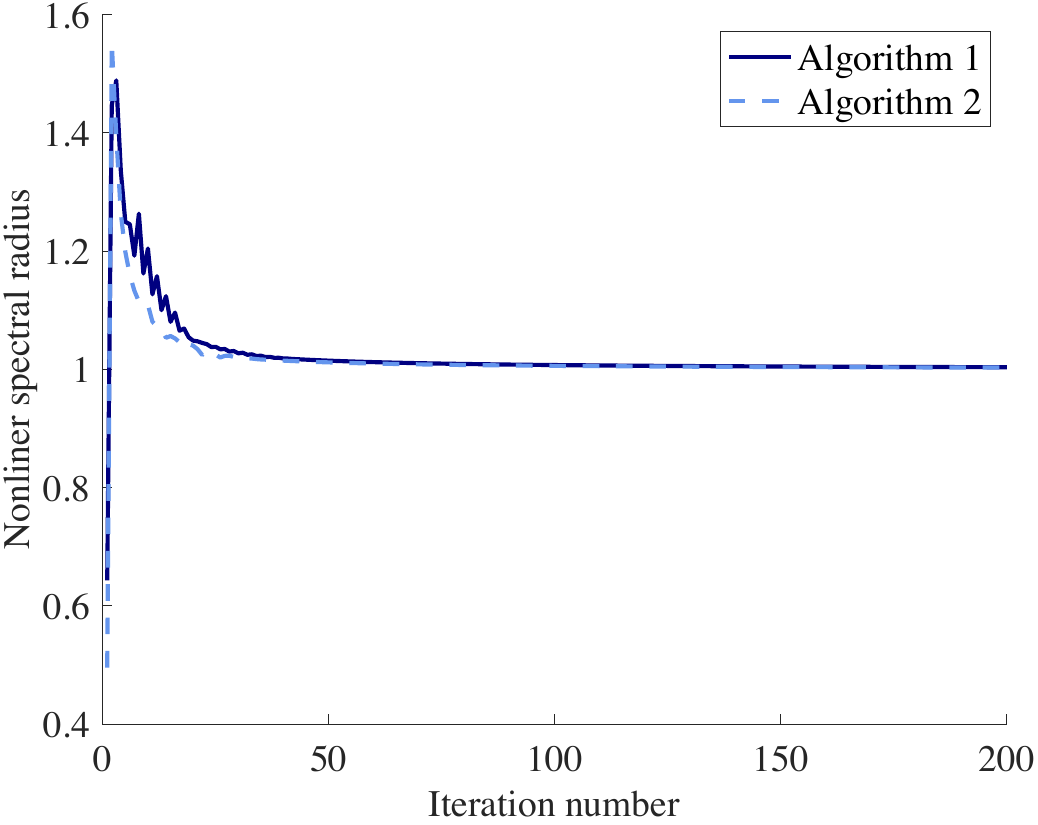}}
\caption{Behavior of algorithms for a simulation where $\bm{M}_l$ is not an inverse Z-matrix for some $l \in \{1,\ldots,L\}$.}
\label{fig.simu.notinvZ}
\end{figure}

\section{Summary}
We have addressed the (weighted) sum-rate maximization problem over the achievable rate region expressed as the nonlinear spectral radius constraint set. This set tends to be convex in UatF-based resource allocation in cell-less networks. Nevertheless, even under convexity, sum-rate maximization is challenging because the projections onto the nonlinear spectral radius constraint sets are difficult to compute. We have resolved this difficulty by exploiting the subgradient projections onto the level sets of the carefully reformulated spectral radius functions. In particular, we have demonstrated that the proposed algorithm provably converges to the global optimum of the sum-rate maximization problem under the typical conditions in cell-less networks.

\appendices
\section{Computing a solution to Problem \eqref{eq.sumrate.power} from a solution to Problem \eqref{eq.sumrate.rate}}
\label{appendix.conversion}
Once a solution to Problem \eqref{eq.sumrate.rate} is obtained,
we can compute a solution to the original sum-rate maximization problem \eqref{eq.sumrate.power} as follows \cite[Sect.~IV.E]{renatoconvex}:

\begin{fact}
\label{fact.conversion}
Suppose that $F\colon \real_+^N\to\real_{++}^N\colon \signal{x}\mapsto (f_n(\signal{x}))_{n \in \mathcal{N}}$ with $f_n~(n\in\mathcal{N})$ in Problem \eqref{eq.sumrate.power} is a SSOC mapping satisfying Assumption \ref{asmp.compatible} with the order-preserving norm $\|\cdot\|$ in \eqref{eq.C}.
Let $F_{\|\cdot\|}$ be the PHOC mapping constructed as shown in Fact \ref{fact.compatible}
from $F$ and $\|\cdot\|$,
$\bm{r}^{\star} \in \mathbb{R}_{+}^{N}$ be a solution to Problem \eqref{eq.sumrate.rate}
using $G=F_{\|\cdot\|}$,
$\mathcal{I} := \{n \in \mathcal{N} \mid r_n^\star > 0\}$,
$\signal{D}^\star := \mathrm{diag}((e^{r_n^\star}-1)_{n \in \mathcal{I}}) \in \mathbb{R}_{++}^{|\mathcal{I}| \times |\mathcal{I}|}$, $F_{\mathcal{I}}\colon\mathbb{R}_{+}^{|\mathcal{I}|}\to\mathbb{R}_{++}^{|\mathcal{I}|}$
be the mapping obtained by extracting the coordinates of $F$ in $\mathcal{I}$,\footnote{For example, if $N = 3$ and $\mathcal{I} = \{1, 3\}$, then $F_{\mathcal{I}}\colon\mathbb{R}_{+}^2\to\mathbb{R}_{++}^2$ is given by $(\forall \bm{q} \in \mathbb{R}_{+}^2)~F_{\mathcal{I}}(\bm{q}) := (f_n(q_1, 0, q_2))_{n \in \mathcal{I}}$.} and
$T_{F, \mathcal{I}, \signal{D}^\star}\colon\mathbb{R}_{+}^{|\mathcal{I}|}\to\mathbb{R}_{++}^{|\mathcal{I}|}\colon \bm{q}\mapsto\signal{D}^\star (F_{\mathcal{I}} (\bm{q}))$. Then $\bm{p}^{\star} \in \mathbb{R}_{+}^{N}$ given by
\begin{align*}
&(p_n^{\star})_{n \in \mathcal{I}} \in \mathrm{Fix}(T_{F, \mathcal{I}, \signal{D}^\star}) =
\{\bm{q} \in \mathbb{R}_{++}^{|\mathcal{I}|} \mid  \signal{D}^\star (F_{\mathcal{I}} (\bm{q})) = \bm{q} \},\\
&p_n^{\star} = 0 \quad (\forall n \in \mathcal{N} \setminus \mathcal{I})
\end{align*}
is a solution to Problem \eqref{eq.sumrate.power}.
\end{fact}

Since $F$ is a SSOC mapping by assumption, so is the mapping $T_{F, \mathcal{I}, \signal{D}^\star}$.
As a result, since $(\spradG \circ \mathfrak{E})(\bm{r}^\star)  \leq 1$ holds
for the solution $\bm{r}^\star$ to Problem \eqref{eq.sumrate.rate} using $G=F_{\|\cdot\|}$, the fixed point of $T_{F, \mathcal{I}, \signal{D}^\star}$ uniquely exists \cite[Sect.~III.C]{renatoconvex}.
Moreover, the uniquely existing fixed point can be computed with the standard fixed point iterations \cite{yates95,nuzman07} or with acceleration methods \cite{cavalcante2016}. If $F$ is also a \emph{positive concave mapping}, convergence is guaranteed to be geometric in the standard Euclidean space (see \cite{piotrowski2022} for details).

\section{Supplementary definitions and facts}
\label{appendix.proof.subgradient}
For completeness, we clarify definitions of differentiability and collect known results that are mainly used in the proof of
Proposition \ref{prop.subgrad.rate.maxlinG}.

\begin{definition}
\label{def.differentiable}
Let $U$ be an open subset of $(\real^N, \langle \cdot, \cdot \rangle, \|\cdot\|_2)$. A mapping $T\colon U \to\real^M$ is
G\^{a}teaux differentiable
at $\signal{x} \in U$ if there exists a linear operator 
$\mathsf{D}T({\bm{x}})\colon \real^N \rightarrow \real^M$ such that
$$
(\forall \bm{d} \in \real^N) \quad \lim_{0 \neq \alpha \to 0} \frac{T(\bm{x} + \alpha\bm{d})-T(\bm{x})}{\alpha} = (\mathsf{D}T({\bm{x}}))(\bm{d}).
$$
A mapping $T\colon U \to\real^M$ is Fr\'{e}chet differentiable
at $\signal{x} \in U$ if there exists a linear operator $\mathsf{D}T({\bm{x}})\colon \real^N \rightarrow \real^M$ such that
$$
\quad \lim_{0 \neq \|\bm{d}\|_2 \to 0} 
\frac{\|T(\bm{x} + \bm{d})-T(\bm{x})- (\mathsf{D}T({\bm{x}}))(\bm{d})\|_2}{\|\bm{d}\|_2} = 0.
$$
If a function $f\colon U\to \real$ is G\^{a}teaux differentiable at $\signal{x} \in U$,
then there exists a unique vector $\nabla f(\bm{x}) \in \real^N$ such that
$(\forall \bm{d} \in \real^N)~(\mathsf{D}T({\bm{x}}))(\bm{d}) = \langle \nabla f(\bm{x}), \bm{d} \rangle$,
and $\nabla f(\bm{x})$ is called the G\^{a}teaux gradient of $f$ at $\bm{x}$.
If a function $f\colon U\to \real$ is Fr\'{e}chet differentiable at $\signal{x} \in U$,
the Fr\'{e}chet gradient of $f$ at $\bm{x} \in U$ is also defined similarly
and denoted by $\nabla f(\bm{x}) \in \real^N$.
Note that G\^{a}teaux and Fr\'{e}chet differentiability are the same notions for convex functions from $\real^N$
to $\real$ (see, e.g., \cite[Corollary~17.44]{baus17}).
\end{definition}

The following fact is a simplification of \cite[Theorem~18.5]{baus17} to convex functions from $\real^N$ to $\real$.

\begin{fact}
\label{fact.subdifmaxconv}
Let $h_l\colon \real^{N}\to\real$ be a convex function for every $l \in \{1,\ldots,L\}$,
and $h\colon \real^{N}\to\real\colon \signal{x}\mapsto \max_{l \in \{1,\ldots,L\}} h_l (\signal{x})$.
For every $\signal{x} \in \real^{N}$,
define
$\mathcal{I}(\signal{x}) := \argmax_{l \in \{1,\ldots,L\} }h_l (\signal{x})$.
Then we have
$(\forall \bm{x} \in \real^N)~\partial h (\signal{x}) = \mathrm{conv}(\cup_{l \in \mathcal{I}(\signal{x})}\partial h_l (\signal{x}))$,
where the convex hull of a set $S\subset\real^N$ is denoted by
$\mathrm{conv}(S)$.
\end{fact}

While the next fact appears (without detailed derivation) in the proof of \cite[Theorem~4.3]{friedland81},
we include a proof in the Supplemental Material for completeness.
\begin{fact}
\label{fact.diff.rhodiagposM}
Let $\signal{M} \in \real_{++}^{N \times N}$.
Then $\varrho_{\signal{M}}\colon \real_{+}^N \to \real_{+}$
given in Definition \ref{def:spradG} is Fr\'{e}chet differentiable at every $\bm{x} \in \real_{++}^{N}$,
and its gradient is given by
\begin{align}
\label{eq.grad.rhodiagposM}
\nabla \varrho_{\signal{M}}(\signal{x}) = 
\frac{\mathrm{diag}(\signal{\eta})\signal{M}\signal{\xi}}{\bm{\eta}^{\mathsf{T}} \bm{\xi}},
\end{align}
where $\signal{\xi} \in \real_{++}^{N}$ and $\signal{\zeta} \in \real_{++}^{N}$
are positive right and left eigenvectors of
$\mathrm{diag}(\signal{x})\signal{M} \in \real_{++}^{N \times N}$
corresponding to the spectral radius $\rho(\mathrm{diag}(\signal{x})\signal{M}) > 0$.
\end{fact}

\bibliographystyle{IEEEtran}
\bibliography{IEEEabrv,references}

\clearpage

\section*{\bf Supplemental Material}

\section{Proof of Fact \ref{fact.diff.rhodiagposM}}
Our proof follows steps similar to the proof of \cite[Theorem 6.3.12]{horn12} that essentially shows
G\^{a}teaux differentiability of a simple eigenvalue as a function of a matrix.
We modify it to prove Fr\'{e}chet differentiability of 
$\varrho_{\signal{M}}\colon \real_{+}^N \to \real_{+}$ (see Definitions \ref{def.nl_radius} and \ref{def:spradG})
 at $\bm{x} \in \real_{++}^N$, provided that $\bm{M}$ is a positive matrix.

Fix $\bm{x} \in \real_{++}^{N}$ arbitrarily. By assumption, we have $\bm{A} := \mathrm{diag}(\signal{x})\signal{M} \in \real_{++}^{N \times N}$.
By standard Perron-Frobenius theory \cite[Theorem 8.2.8]{horn12},
$\lambda := \rho(\bm{A}) > 0$ is an algebraically simple eigenvalue of
$\signal{A}$, and there exist
positive right and left eigenvectors $\bm{\xi}$ and $\bm{\eta}$ corresponding to 
$\lambda = \rho(\bm{A})$. Hence, by \cite[Lemma 6.3.10]{horn12}, there exist $\bm{S}_1 \in \mathbb{C}^{N \times (N-1)}$ and $\bm{Z}_1 \in \mathbb{C}^{N \times (N-1)}$ such that $\bm{S} := [\bm{\xi}~\bm{S}_1] \in \mathbb{C}^{N \times N}$ is nonsingular, its inverse is given by
$\bm{S}^{-1} = [\tilde{\bm{\eta}}~\bm{Z}_1]^{\mathsf{H}}$ with $\tilde{\bm{\eta}} := \bm{\eta}/(\bm{\eta}^{\mathsf{T}}\bm{\xi})$, and
\begin{align*}
\bm{S}^{-1}\bm{A}\bm{S} =
\begin{bmatrix}
\lambda &\bm{0}^{\mathsf{T}}\\
\bm{0} &\bm{A}_1
\end{bmatrix}.
\end{align*}
Since $\lambda$ is an algebraically simple eigenvalue of $\signal{A}$, it is not an eigenvalue of $\bm{A}_1 \in \mathbb{C}^{(N-1)\times(N-1)}$.
Hence, by \cite[Theorem 8.2.8(e)]{horn12}, there exists $\mu > 0$ such that
\begin{align}
\label{eq.gap.sprad}
\lambda \geq |\hat{\lambda}| + \mu \quad \mbox{if $\hat{\lambda} \in \mathbb{C}$ is an eigenvalue of $\bm{A}_1$.}
\end{align}
Fix $\varepsilon \in (0, \mu/7)$ arbitrarily. Then, by \cite[Theorem 2.4.72]{horn12},
there exists a nonsingular matrix $\bm{S}_{\varepsilon} \in \mathbb{C}^{(N-1) \times (N-1)}$
such that $\bm{T}_{\varepsilon} := \bm{S}_{\varepsilon}^{-1}\bm{A}_1\bm{S}_{\varepsilon} \in \mathbb{C}^{(N-1) \times (N-1)}$ is upper triangular, the eigenvalues of $\bm{A}_1$ are on the main diagonal of $\bm{T}_{\varepsilon}$, and
\begin{align}
\label{eq.offdiagTeps}
(\forall i \in \{1,\ldots,N-2\})\quad \sum_{j = i+1}^{N-1} |t_{i,j}^{(\varepsilon)}| \leq \varepsilon,
\end{align}
where $t_{i,j}^{(\varepsilon)}$ denotes the $(i,j)$ entry of $\bm{T}_{\varepsilon}$.

Let $\bm{d} \in \real^{N}$. Since $\varrho_{\signal{M}}(\signal{x}+\bm{d}) = \rho(\bm{A}+\mathrm{diag}(\bm{d})\bm{M})$, we consider the eigenvalues of $\bm{A}+\mathrm{diag}(\bm{d})\bm{M}$. Performing the similarity transformation with $\bm{S}$ and $\bm{S}_{\varepsilon}$, we obtain
\begin{multline*}
\begin{bmatrix}
1 &\bm{0}^{\mathsf{T}}\\
\bm{0} &\bm{S}_{\varepsilon}^{-1}
\end{bmatrix}
\bm{S}^{-1}(\bm{A}+\mathrm{diag}(\bm{d})\bm{M})\bm{S}
\begin{bmatrix}
1 &\bm{0}^{\mathsf{T}}\\
\bm{0} &\bm{S}_{\varepsilon}
\end{bmatrix}\\
=
\begin{bmatrix}
\lambda + \bm{d}^{\mathsf{T}}\mathrm{diag}(\tilde{\bm{\eta}})\bm{M}\bm{\xi}
&\tilde{\bm{\eta}}^{\mathsf{T}}\bm{D}\bm{M}\bm{S}_1\bm{S}_{\varepsilon}\\
\bm{S}_{\varepsilon}^{-1}\bm{Z}_1^{\mathsf{H}}\bm{D}\bm{M}\bm{\xi}
&\bm{T}_{\varepsilon} + \bm{S}_{\varepsilon}^{-1}\bm{Z}_1^{\mathsf{H}}\bm{D}\bm{M}\bm{S}_1\bm{S}_{\varepsilon}
\end{bmatrix},
\end{multline*}
where we let $\bm{D} := \mathrm{diag}(\bm{d}) \in \mathbb{R}^{N \times N}$.
Let $\| \cdot \|_{\infty}$ denotes the $\ell_{\infty}$ operator norm, i.e., for every $\bm{X} =[x_{i,j}] \in \mathbb{C}^{I \times J}$,
\begin{align*}
\|\bm{X}\|_{\infty} := \max_{\bm{u} \in \mathbb{C}^{J}\setminus\{\bm{0}\}} \frac{\|\bm{X}\bm{u}\|_{\infty}}{\|\bm{u}\|_{\infty}} = \max_{i \in \{1,\ldots,I\}} \sum_{j=1}^{J} |x_{i,j}|.
\end{align*}
By equivalence of norms on $\real^N$, there exists $C > 0$ such that
\begin{align}
\label{equiv.norm}
(\forall \bm{u} \in \mathbb{R}^{N})\quad \|\bm{u}\|_{\infty} \leq C \|\bm{u}\|_2.
\end{align}
For the arbitrarily fixed $\varepsilon \in (0, \mu/7)$, there exists $r > 0$ such that
\begin{align}
\label{ieq.bound.r}
rC\|\tilde{\bm{\eta}}^{\mathsf{T}} \|_{\infty}\| \bm{M}\bm{S}_1\bm{S}_{\varepsilon} \|_{\infty} < \varepsilon.
\end{align}
Note that we can choose such $r$ independently of $\bm{d}$ because $\tilde{\bm{\eta}}$ and $\bm{M}\bm{S}_1\bm{S}_{\varepsilon}$ are independent of $\bm{d}$. We perform the final similarity transformation:
\begin{multline*}
\begin{bmatrix}
1 &\bm{0}^{\mathsf{T}}\\
\bm{0} &r^{-1}\bm{S}_{\varepsilon}^{-1}
\end{bmatrix}
\bm{S}^{-1}(\bm{A}+\mathrm{diag}(\bm{d})\bm{M})\bm{S}
\begin{bmatrix}
1 &\bm{0}^{\mathsf{T}}\\
\bm{0} &r\bm{S}_{\varepsilon}
\end{bmatrix}\\
=
\begin{bmatrix}
\lambda + \bm{d}^{\mathsf{T}}\mathrm{diag}(\tilde{\bm{\eta}})\bm{M}\bm{\xi}
&r\tilde{\bm{\eta}}^{\mathsf{T}}\bm{D}\bm{M}\bm{S}_1\bm{S}_{\varepsilon}\\
r^{-1}\bm{S}_{\varepsilon}^{-1}\bm{Z}_1^{\mathsf{H}}\bm{D}\bm{M}\bm{\xi}
&\bm{T}_{\varepsilon} + \bm{S}_{\varepsilon}^{-1}\bm{Z}_1^{\mathsf{H}}\bm{D}\bm{M}\bm{S}_1\bm{S}_{\varepsilon}
\end{bmatrix} =: \bm{B}.
\end{multline*}

Now we consider the Ger\v{s}gorin discs associated with the matrix $\bm{B}$. Independently of $\bm{d}$, we can choose $\delta_1 > 0$, $\delta_2 > 0$, and $\delta_3 > 0$ such that 
\begin{align}
\label{ieq.delta1}
\delta_1\|\mathrm{diag}(\tilde{\bm{\eta}})\bm{M}\bm{\xi} \|_2 < \varepsilon,\\
\label{ieq.delta2}
\delta_2 C\|r^{-1}\bm{S}_{\varepsilon}^{-1}\bm{Z}_1^{\mathsf{H}} \|_{\infty} \| \bm{M}\bm{\xi}\|_{\infty} < \varepsilon,\\
\label{ieq.delta3}
\delta_3 C\| \bm{S}_{\varepsilon}^{-1}\bm{Z}_1^{\mathsf{H}} \|_{\infty} \|\bm{M}\bm{S}_1\bm{S}_{\varepsilon} \|_{\infty} < \varepsilon.
\end{align}
In addition, since $\bm{A}$ is positive, we can choose $\delta_4 > 0$ such that,
 if $\bm{d} \in \mathbb{R}^{N}$ satisfies $\|\bm{d}\|_2 < \delta_4$, then
$\bm{A}+\mathrm{diag}(\bm{d})\bm{M}$ is also positive.
Let $\delta := \min\{\delta_1,\delta_2,\delta_3,\delta_4, 1\}$, and suppose that
$\bm{d} \in \mathbb{R}^{N}$ satisfies $\|\bm{d}\|_2 < \delta$.
Then, the Ger\v{s}gorin disc $\mathcal{G}_1$ associated with the first row of the matrix $\bm{B} = [b_{i,j}]$ satisfies
\begin{align}
\nonumber
\mathcal{G}_1 :={}& \{\bm{z} \in \mathbb{C} \mid |z - b_{1,1} | \leq \|r\tilde{\bm{\eta}}^{\mathsf{T}}\bm{D}\bm{M}\bm{S}_1\bm{S}_{\varepsilon}\|_{\infty} \}\\
\label{eq.bound1.gdisc1}
\subset {}& \{\bm{z} \in \mathbb{C} \mid |z - b_{1,1} | < \varepsilon \|\bm{d}\|_{2} \}
\end{align}
because
\begin{multline}
\label{derivation.bound.linfnorm}
\|r\tilde{\bm{\eta}}^{\mathsf{T}}\bm{D}\bm{M}\bm{S}_1\bm{S}_{\varepsilon}\|_{\infty} 
\overset{\text{(a)}}\leq r \|\tilde{\bm{\eta}}^{\mathsf{T}} \|_{\infty}\|\bm{D}\|_{\infty} \| \bm{M}\bm{S}_1\bm{S}_{\varepsilon} \|_{\infty}\\
\overset{\text{(b)}}\leq rC\|\tilde{\bm{\eta}}^{\mathsf{T}} \|_{\infty}\| \bm{M}\bm{S}_1\bm{S}_{\varepsilon} \|_{\infty} \|\bm{d} \|_2
\overset{\text{(c)}}< \varepsilon \|\bm{d} \|_2,
\end{multline}
where (a) follows from submultiplicativity of the $\ell_{\infty}$ operator norm,
(b) from  $\|\bm{D}\|_{\infty} = \|\bm{d}\|_{\infty}$ and \eqref{equiv.norm}, and (c) from \eqref{ieq.bound.r}.
Furthermore, since $b_{1,1} = \lambda + \bm{d}^{\mathsf{T}}\mathrm{diag}(\tilde{\bm{\eta}})\bm{M}\bm{\xi}$ by construction, we have
$$
|b_{1,1} - \lambda | = |\bm{d}^{\mathsf{T}}\mathrm{diag}(\tilde{\bm{\eta}})\bm{M}\bm{\xi}|
\overset{\text{(a)}}\leq \|\bm{d}\|_2\|\mathrm{diag}(\tilde{\bm{\eta}})\bm{M}\bm{\xi}\|_2 \overset{\text{(b)}}< \varepsilon,
$$
where (a) follows from the Cauchy-Schwarz inequality and (b) from $\|\bm{d}\|_2 < \delta_1$ and \eqref{ieq.delta1}.
Combining this inequality with \eqref{eq.bound1.gdisc1}, since $\|\bm{d}\|_2 < 1$ holds by assumption, we obtain
\begin{align}
\label{eq.bound2.gdisc1}
\mathcal{G}_1 \subset {}& \{\bm{z} \in \mathbb{C} \mid |z - \lambda | < 2 \varepsilon \}.
\end{align}
Next, for each $i \in \{ 2,3,\ldots,N\}$, we consider the Ger\v{s}gorin disc $\mathcal{G}_i$ associated with the $i$th row of $\bm{B}$.
Similarly to the derivation in \eqref{derivation.bound.linfnorm}, using \eqref{ieq.delta2} and $\|\bm{d}\|_2 < \delta_2$, we obtain
\begin{multline*}
\|r^{-1}\bm{S}_{\varepsilon}^{-1}\bm{Z}_1^{\mathsf{H}}\bm{D}\bm{M}\bm{\xi} \|_{\infty}\\
\leq  \|\bm{d}\|_2 C\|r^{-1}\bm{S}_{\varepsilon}^{-1}\bm{Z}_1^{\mathsf{H}} \|_{\infty} \| \bm{M}\bm{\xi}\|_{\infty} < \varepsilon.
\end{multline*}
Similarly, by \eqref{ieq.delta3} and $\|\bm{d}\|_2 < \delta_3$, we have
\begin{multline*}
\|\bm{S}_{\varepsilon}^{-1}\bm{Z}_1^{\mathsf{H}}\bm{D}\bm{M}\bm{S}_1\bm{S}_{\varepsilon}\|_{\infty}\\
\leq \|\bm{d}\|_2 C\| \bm{S}_{\varepsilon}^{-1}\bm{Z}_1^{\mathsf{H}} \|_{\infty} \|\bm{M}\bm{S}_1\bm{S}_{\varepsilon} \|_{\infty}
< \varepsilon.
\end{multline*}
In particular, this inequality implies that any entry of $\bm{S}_{\varepsilon}^{-1}\bm{Z}_1^{\mathsf{H}}\bm{D}\bm{M}\bm{S}_1\bm{S}_{\varepsilon}$ is less than $\varepsilon$, and hence
$$
(\forall i \in \{2,3, \ldots,N\})\quad |b_{i,i} - t_{i-1,i-1}^{(\varepsilon)}| < \varepsilon.
$$
We recall that $\bm{T}_{\varepsilon}$ is upper triangular, and the inequality shown in \eqref{eq.offdiagTeps} holds.
Altogether, for every $i \in \{ 2,3,\ldots,N\}$, we obtain
\begin{align}
\nonumber
\mathcal{G}_i :={}& \left\{\bm{z} \in \mathbb{C}~\left|~|z - b_{i,i} | \leq \sum_{j \in \{1,\ldots,N\}\setminus \{i\} } |b_{i,j}| \right.\right\}\\
\nonumber
\subset {}& \{\bm{z} \in \mathbb{C} \mid |z - b_{i,i} | < 3 \varepsilon \}\\
\label{eq.bound.gdisci}
\subset {}& \{\bm{z} \in \mathbb{C} \mid |z - t_{i-1,i-1}^{(\varepsilon)} | < 4 \varepsilon \}.
\end{align}
For every $i \in \{ 2,3,\ldots,N\}$,
since $t_{i-1,i-1}^{(\varepsilon)}$ is an eigenvalue of $\bm{A}_1$ by construction, using the inequality shown in \eqref{eq.gap.sprad},
we derive
$$
|\lambda - t_{i-1,i-1}^{(\varepsilon)}| \geq \lambda - |t_{i-1,i-1}^{(\varepsilon)}| \geq \mu.
$$
Combining this inequality with \eqref{eq.bound2.gdisc1} and \eqref{eq.bound.gdisci}, since $6\varepsilon < \mu$ holds by assumption, we deduce that the Ger\v{s}gorin disc $\mathcal{G}_1$ is disjoint from the other discs $\mathcal{G}_2,\ldots,\mathcal{G}_N$.
Hence, for every $\bm{d} \in \mathbb{R}^N$ satisfying $\|\bm{d}\|_2 < \delta$, by \eqref{eq.bound1.gdisc1} and the Ger\v{s}gorin theorem \cite[Theorem 6.1.1]{horn12}, there exists a unique eigenvalue,
say $\lambda(\bm{d})$, of $\bm{A}+\mathrm{diag}(\bm{d})\bm{M}$ such that $\lambda(\bm{d}) \in \mathcal{G}_1$ and
\begin{align}
\label{eq.lambda.linear.approx}
|\lambda(\bm{d}) - \lambda - \bm{d}^{\mathsf{T}}\mathrm{diag}(\tilde{\bm{\eta}})\bm{M}\bm{\xi}  | < \varepsilon\|\bm{d}\|_2.
\end{align}

Furthermore, under the assumed conditions, we show that $\lambda(\bm{d})$ is the spectral radius of $\bm{A}+\mathrm{diag}(\bm{d})\bm{M}$ as follows.
Let $\hat{\lambda}(\bm{d})$ be an eigenvalue of $\bm{A}+\mathrm{diag}(\bm{d})\bm{M}$ such that $\hat{\lambda}(\bm{d}) \neq \lambda(\bm{d})$. Then, since $\hat{\lambda}(\bm{d})$ should be contained in one of $\mathcal{G}_2,\ldots,\mathcal{G}_N$, the relation in \eqref{eq.bound.gdisci} implies
$$
(\exists i \in \{2,3, \ldots,N\}) \quad |\hat{\lambda}(\bm{d}) - t_{i-1,i-1}^{(\varepsilon)}| < 4\varepsilon.
$$
Hence, we have
$$
|\hat{\lambda}(\bm{d})| < |t_{i-1,i-1}^{(\varepsilon)}| + 4\varepsilon \overset{\text{(a)}}< \lambda - 2\varepsilon \overset{\text{(b)}}< |\lambda(\bm{d})|,
$$
where (a) holds by \eqref{eq.gap.sprad} and $6\varepsilon < \mu$,
and (b) by $\lambda(\bm{d}) \in \mathcal{G}_1$ and \eqref{eq.bound2.gdisc1}.
We recall that $\bm{A}+\mathrm{diag}(\bm{d})\bm{M}$ becomes positive
if $\|\bm{d}\|_2 < \delta_4$.
Altogether, for every $\bm{d} \in \mathbb{R}^{N}$ satisfying $\|\bm{d}\|_2 < \delta$, 
we have $\lambda(\bm{d}) = \rho(\bm{A} + \mathrm{diag}(\bm{d}) \bm{M}) > 0$ by standard Perron-Frobenius theory.

Finally, the desired result is shown by
\begin{multline*}
\lim_{0 \neq \|\bm{d}\|_2 \rightarrow 0}
\frac{\varrho_{\signal{M}}(\signal{x}+\bm{d})
- \varrho_{\signal{M}}(\signal{x})
-\langle\bm{d}, \mathrm{diag}(\tilde{\bm{\eta}})\bm{M}\bm{\xi} \rangle}
{\|\bm{d}\|_2}\\
= \lim_{0 \neq \|\bm{d}\|_2 \rightarrow 0}
\frac{\rho(\bm{A} + \mathrm{diag}(\bm{d}) \bm{M})
- \rho(\bm{A})
-\bm{d}^\mathsf{T} \mathrm{diag}(\tilde{\bm{\eta}})\bm{M}\bm{\xi}}
{\|\bm{d}\|_2}\\
= \lim_{0 \neq \|\bm{d}\|_2 \rightarrow 0}
\frac{\lambda(\bm{d})
- \lambda
-\bm{d}^\mathsf{T} \mathrm{diag}(\tilde{\bm{\eta}})\bm{M}\bm{\xi}}
{\|\bm{d}\|_2} = 0,
\end{multline*}
where the last equality follows from \eqref{eq.lambda.linear.approx} since $\varepsilon$ can be chosen from $(0, \mu/7)$ arbitrarily.

\section{Optimization algorithm for Problem \eqref{eq.sumrate.sinr}}
We present an efficient algorithm that provably converges to the global optimal value of Problem \eqref{eq.sumrate.sinr} under the following assumption.

\begin{assumption}
	\label{asmp.norm_inducing}
	In addition to Assumption \ref{asmp.G.hypTnorm},	
	 we assume that the function $\spradG\colon\real_{+}^{N}\to\real_{+}$ given in Definition \ref{def:spradG} is convex on $\real_{++}^N$.
\end{assumption}
Note that this assumption is stronger than Assumption \ref{asmp.convex.extension.rho.exp}
because convexity of $\spradG$ on $\real_{++}^N$
implies convexity of $\spradG\circ\mathfrak{E}\colon\real_{+}^{N}\to\real_{+}$ by \cite[Proposition 6]{renatoconvex}.

\subsection{Known results used in algorithm derivation}
In this subsection, we present the known results
used to derive the proposed algorithm for Problem \eqref{eq.sumrate.sinr}.

\begin{fact}[{\cite[Proposition 3]{renatoconvex}}]
	\label{fact.norm}
	Let a PHOC mapping $G\colon \real_{+}^N\to\real_{+}^N$ satisfy Assumption \ref{asmp.norm_inducing}.
    Then, for $C_G:=\{\signal{x}\in\real^N_{+}~\mid \spradG(\bm{x})\le 1\}$ and $S=\mathrm{conv}(C_G\cup -C_G)$, where $-C_G:=\{\signal{x}\in\real^N\mid -\signal{x}\in C_G\}$, the Minkowski or gauge functional of $S$, defined by
	\begin{align}
		\label{eq.minkowski}
		(\forall\signal{x}\in\real^N)\quad\|\signal{x}\|_G := \inf\{\gamma>0\mid (1/\gamma)\signal{x}\in S_G\},
	\end{align}
	is an order-preserving norm satisfying
	\begin{align}
	\label{eq.extention.rhodiagG}
	(\forall\signal{x}\in\real_{+}^N)\quad
	\|\signal{x}\|_G = \spradG(\bm{x}).
	\end{align}
\end{fact}

\begin{fact}[{\cite[Example 2]{renatoconvex}}]
\label{fact.maxrhoconvex}
Consider the setting of Example \ref{ex.maxlinearG}.
For every $l \in \{1,\ldots,L\}$,
let $G_l\colon \real_+^N\to\real_{+}^N\colon \signal{x}\mapsto \signal{M}_l\signal{x}$,
and suppose that the function $\varrho_{\signal{M}_l}\colon\real_{+}^{N}\to\real_{+}$ is convex on $\real_{++}^N$.
Then $G_l$ satisfies the assumptions in Fact \ref{fact.norm},
and hence
\begin{align}
\label{eq.normG.specradlinar}
(\forall\signal{x}\in\real_+^N)\quad\|\signal{x}\|_{G_{l}} = \varrho_{\signal{M}_l}(\signal{x})
\end{align}
holds for every $l \in \{1,\ldots,L\}$, where $\|\cdot\|_{G_{l}}$ is constructed as in \eqref{eq.minkowski}.
Moreover, $G$ given in Example \ref{ex.maxlinearG} satisfies Assumption \ref{asmp.norm_inducing}, and
we have
\begin{align}
\label{eq.max.Gnorm}
(\forall\signal{x}\in\real^{N})\quad\|\signal{x}\|_{G} = \max_{l\in\{1,\ldots,L\}} \|\signal{x}\|_{G_l}.
\end{align}
\end{fact}

\subsection{Reformulation to a problem over Euclidean space}
To apply the HSD method, under Assumption \ref{asmp.norm_inducing}, we translate Problem \eqref{eq.sumrate.sinr} into an equivalent convex optimization problem
composed of convex functions from $\real^N$ to $\real$.
If $G$ satisfies Assumption \ref{asmp.norm_inducing},
then by Fact \ref{fact.norm}, the function $\spradG\colon \real_{+}^{N}\rightarrow\real_{+}$ can be extended to
the norm $\|\cdot\|_{G}$ given by \eqref{eq.minkowski}, which coincides with $\spradG$ on $\real_{+}^{N}$.
Meanwhile, the value of the cost function $-\sum_{n \in \mathcal{N}}w_n\log(1 + s_n)$ in Problem \eqref{eq.sumrate.sinr}
is well defined if $1+s_n > 0$ for every $n \in \mathcal{N}$.
Note that, since the constraint set $\mathcal{S}$ is a subset of $\real_{+}^{N}$ (see \eqref{eq.def.regionsinr}),
it is sufficient if the extended function
coincides with $-\sum_{n \in \mathcal{N}}w_n\log(1 + s_n)$ on $\real_{+}^{N}$.
Specifically, we use the following simple extension.
\begin{lemma}
\label{lemma.LipGradCost.sinr}
With $\signal{w} \in \real_{++}^{N}$, we define
	\begin{align*}
		&\Phi\colon \real^{N}\to\real: \signal{s} \mapsto \sum_{n \in \mathcal{N}}w_n\phi(s_n),\\
		&\phi\colon \real \to \real : s \mapsto \begin{cases}
					-\log (1+s), &\text{if}~ s \ge 0;\\
					-s, &\text{if}~s < 0.
				\end{cases}
	\end{align*}
Then $\Phi$ satisfies the following properties.
\begin{enumerate}[\itemsep=2pt]
 \item[(i)] $(\forall \signal{s} \in \real_{+}^{N})~\Phi(\signal{s}) = -\sum_{n \in \mathcal{N}}w_n\log(1 + s_n)$;
 
\item[(ii)] $\Phi$ is differentiable and convex on $\real^{N}$, and strictly convex on $\real_{+}^{N}$; and

\item[(iii)] The gradient of $\Phi$ on $\real_{+}^{N}$ is given by
\begin{align}
\label{eq.grad.sumrate.sinr}
(\signal{s} \in \real_{+}^{N})\quad \nabla\Phi(\signal{s}) = -\left(\frac{w_n}{1+s_n}\right)_{n \in \mathcal{N}},
\end{align}
and it is Lipschitz continuous on $(\real_{+}^{N}, \|\cdot\|_2)$.
\end{enumerate}
\end{lemma}
\begin{proof}

(i) Clear from the definition of $\Phi$.

(ii) It is clear that $\phi$ is continuously differentiable at $s \in \real \setminus \{0\}$.
Since the left and right derivatives of $\phi$ at $0$ are the same,
we deduce that $\phi$ is continuously differentiable on $\real$, implying that
$\Phi$ is Fr\'{e}chet differentiable on $\real^{N}$.
The derivative of $\phi$ is given by
\begin{align}
\label{eq.derivative.phi}
\quad \phi' (s) := \begin{cases}
					-1/(1+s), &\text{if}~ s \ge 0;\\
					-1, &\text{if}~s < 0.
				\end{cases}
\end{align}
Since $\phi'$ is increasing on $\real$, $\phi$ is convex on $\real$, and so is $\Phi$ on $\real^{N}$
because of $\signal{w} \in \real_{++}^{N}$.
Strict convexity of $\Phi$ on $\real_{+}^{N}$ is clear from the definition.

(iii) The expression in \eqref{eq.grad.sumrate.sinr} is immediate from \eqref{eq.derivative.phi}.
For every $\signal{x}$ and $\signal{y}$ in $\real_{+}^{N}$, we have
	\begin{align*}
		&\|\nabla\Phi(\signal{x}) - \nabla\Phi(\signal{y})  \|_2^2 
		=\sum_{n \in \mathcal{N}}\left|-\frac{w_n}{1+x_n} + \frac{w_n}{1+y_n}\right|^2\\
		&=\sum_{n \in \mathcal{N}}\left|\frac{w_n(x_n-y_n)}{(1+x_n)(1+y_n)}\right|^2
		\le\sum_{n \in \mathcal{N}}w_n^2 |x_n-y_n|^2,
	\end{align*}
showing that $\nabla\Phi$ is Lipschitz continuous on $(\real_{+}^{N}, \|\cdot \|_2)$
with Lipschitz constant $\kappa = \max_{n \in \mathcal{N}} w_n$.
\end{proof}

Under Assumption \ref{asmp.norm_inducing},
using $\|\cdot\|_G$ given in Fact \ref{fact.norm} and $\Phi$ given in Lemma \ref{lemma.LipGradCost.sinr},
we reformulate Problem \eqref{eq.sumrate.sinr} to
 \begin{align}
	\label{eq.refsumrate.frsinr}\tag{S'}
	\begin{array}{cl}
		\text{miminize} &  \Phi(\signal{s})  \\
		\text{subject~to} & \signal{s}\in [0, b]^{N}\cap \mathrm{lev}_{\le 1}(\|\cdot\|_G),
	\end{array}
\end{align}
where we also introduce $[0, b]^{N}$ similarly to the design of Problem \eqref{eq.refsumrate.frrate}.
In the following proposition, under Assumption \ref{asmp.norm_inducing}, we  show that
Problems \eqref{eq.sumrate.sinr} and \eqref{eq.refsumrate.frsinr}
are guaranteed to
have the same unique solution if $b$ is large enough.

\begin{proposition}
\label{prop.existence.equivalence.prob.sinr.box}
Under Assumption \ref{asmp.norm_inducing},
Problem \eqref{eq.refsumrate.frsinr} has a unique solution for any $b > 0$.
Moreover, if $b$ is large enough,
the unique solution to Problem \eqref{eq.refsumrate.frsinr}
is the same as the unique solution to Problem \eqref{eq.sumrate.sinr}.
\end{proposition}
\begin{proof}
By Fact~\ref{fact.norm}, for any $b > 0$,
$[0, b]^{N} \cap \mathrm{lev}_{\le 1}(\|\cdot\|_G)$ is a nonempty compact convex set.
Hence, by Lemma \ref{lemma.LipGradCost.sinr}(ii),
Problem \eqref{eq.refsumrate.frsinr}
is minimization of the strictly convex function over the nonempty compact convex set,
implying that a solution uniquely exists for Problem \eqref{eq.refsumrate.frsinr}.
Next, we show that the solution sets of Problems \eqref{eq.sumrate.sinr}
and \eqref{eq.refsumrate.frsinr} are the same if $b$ is sufficiently large.
Since the constraint set $\mathcal{S}$ defined by \eqref{eq.def.regionsinr}
is bounded by Lemma \ref{lemm.bounded} and a subset of $\real_{+}^{N}$,
there exists $c > 0$ such that
\begin{align*}
(\forall b \ge c)\quad \mathcal{S} = [0, b]^{N} \cap \mathcal{S}
= [0, b]^{N}\cap\mathrm{lev}_{\le 1}(\|\cdot\|_G),
\end{align*}
where the last equality holds by \eqref{eq.extention.rhodiagG} in Fact~\ref{fact.norm}.
By this equality and Lemma \ref{lemma.LipGradCost.sinr}(i), we
obtain the desired result.
\end{proof}

\subsection{Optimization algorithm}
We derive the proposed algorithm for Problem \eqref{eq.sumrate.sinr}
by applying the subgradient-projection-based HSD method shown in Fact \ref{fact.hsdmsp} to Problem \eqref{eq.refsumrate.frsinr}.
With an initial point $\signal{s}_{1} \in \real^{N}_{++}$
and a slowly decreasing step-size
$(\mu_k)_{k\in\mathbb{N}} \subset \real_{++}^{N}$,
the proposed
algorithm for Problem \eqref{eq.refsumrate.frsinr} generates the sequence $(\signal{s}_{k})_{k \in \mathbb{N}}$ by
\begin{align}
\label{eq:HSDMmaxsumrate.sinr}
\nonumber
&\text{For}~k = 1, 2, \ldots\\
&\left\lfloor\begin{aligned}
&\text{Choose}~\signal{g}_k \in \partial \|\cdot\|_G(\signal{s}_k),\\
&\hat{\signal{s}}_{k} := \begin{cases}
\bm{s}_{k} - (\|\signal{s}_k\|_G-1)\signal{g}_k/\|\signal{g}_k\|_2^2,
& \text{if}~ \|\signal{s}_k\|_G > 1;\\
\bm{s}_{k},
& \text{if}~ \|\signal{s}_k\|_G \le 1,\\
\end{cases}\\
&\tilde{\signal{s}}_{k} := P_{[0,b]^{N}}(\hat{\signal{s}}_{k}),\\
&\signal{s}_{k+1} := \tilde{\signal{s}}_{k} - \mu_{k}\nabla\Phi(\tilde{\signal{s}}_{k}).
\end{aligned}\right.
\end{align}
An efficient method for finding $\bm{g}_k$ in the scenarios considered in Example \ref{ex.maxlinearG} is given in the next subsection.
The projection mapping $P_{[0, b]^{N}}\colon \real^{N}\to[0, b]^{N}$ can be computed by
\eqref{eq.def.projectionbox}.
For every $k \in \mathbb{N}$,
since $\tilde{\signal{s}}_{k} = P_{[0,b]^{N}}(\hat{\signal{s}}_{k}) \in \real_{+}^{N}$,
the gradient $\nabla\Phi(\tilde{\signal{s}}_{k})$ can be computed by the expression in \eqref{eq.grad.sumrate.sinr}.

Under Assumption \ref{asmp.norm_inducing},
we establish the convergence of $(\signal{s}_{k})_{k \in \mathbb{N}}$
to the unique solution $\signal{s}^{\star}$ to Problem \eqref{eq.refsumrate.frsinr}
in the following theorem, where we also show
that the sum-rate computed for SINR values $(\signal{s}_{k})_{k \in \mathbb{N}}$ converges
to (the negative of) the global optimal value of Problem
\eqref{eq.sumrate.sinr} -- and hence that of the original
problem \eqref{eq.sumrate.power} -- if $b$ is large enough.

\begin{theorem}
\label{theo.converge.sinr}
Let $\signal{w} \in \real_{++}^{N}$,
$b > 0$,
$\signal{s}_1 \in \real_{++}^{N}$, and
$(\mu_k)_{k\in\mathbb{N}} \subset \real_{++}^{N}$ satisfy
$\lim_{k\to\infty}\mu_k = 0$ and
$\sum_{k\in\mathbb{N}}\mu_k = \infty$.
Then, under Assumption \ref{asmp.norm_inducing}, each of the following holds for $(\signal{s}_{k})_{k \in \mathbb{N}}$ generated by \eqref{eq:HSDMmaxsumrate.sinr}.
\begin{enumerate}[\itemsep=2pt]
\item[(i)] $\lim_{k \rightarrow \infty} \signal{s}_{k} = \signal{s}^{\star} \in \argmin_{\signal{s} \in [0, b]^{N}\cap\mathrm{lev}_{\le 1}(\|\cdot\|_G)}\Phi(\signal{s})$;
\item[(ii)] $(\forall k \in \mathbb{N})~\signal{s}_{k} \in \real_{++}^{N}$;
\item[(iii)] $\lim_{k \rightarrow \infty} d_2(\signal{s}_{k}, \mathcal{S}) = 0$; and
\item[(iv)] If $b$ is large enough, then
$$
\lim_{k \rightarrow \infty}\sum_{n \in \mathcal{N}}w_n\log(1 + s_{k,n})
= \max_{\signal{s} \in \mathcal{S}}\sum_{n \in \mathcal{N}}w_n\log(1 + s_{n}),
$$
where we denote $\signal{s}_k = (s_{k,n})_{n \in \mathcal{N}}$ for every $k \in \mathbb{N}$.
\end{enumerate}
\end{theorem}
\begin{proof}
(i) To prove this claim via application of Fact \ref{fact.hsdmsp},
we first show that conditions (a1) and (a2) in Fact \ref{fact.hsdmsp} are satisfied for
$h = \|\cdot\|_G$, $\mathcal{K} = [0, b]^{N}$, and $\Theta = \Phi$.
Since $\|\cdot\|_G$ is a norm by Fact~\ref{fact.norm}, it is a convex function from $\real^{N}$ to $\real$.
For any $b > 0$, $[0, b]^{N}$ is a compact convex set,
and $[0, b]^{N} \cap \mathrm{lev}_{\le 1}(\|\cdot\|_G) \neq \varnothing$.
Thus, condition (a1) holds.
By Lemma \ref{lemma.LipGradCost.sinr}, $\Phi$ satisfies condition (a2).
Since the solution set of \eqref{eq:HSDMmaxsumrate.sinr}
is singleton by Proposition \ref{prop.existence.equivalence.prob.sinr.box},
Fact \ref{fact.hsdmsp} implies $\lim_{k \to \infty}\signal{s}_k = \signal{s}^{\star}$.

(ii) For every $k \in \mathbb{N}$, we have
$\tilde{\signal{s}}_{k} = P_{[0,b]^{N}}(\hat{\signal{s}}_{k}) \in \mathbb{R}_{+}^{N}$.
By assumption, we have $$(\forall \signal{s} \in \mathbb{R}_{+}^{N})\quad -\nabla\Phi(\signal{s}) = \left(\frac{w_n}{1+s_n}\right)_{n \in \mathcal{N}} \in \mathbb{R}_{++}^{N}$$ and
$(\mu_k)_{k\in\mathbb{N}} \subset \real_{++}^{N}$.
Hence, the last step in \eqref{eq:HSDMmaxsumrate.sinr} ensures
$\signal{s}_{k+1} \in \real_{++}^{N}$ for every $k \in \mathbb{N}$.
Note that $\signal{s}_{1} \in \real_{++}^{N}$ holds by assumption.

(iii) Since $\mathcal{S} = \real_{+}^{N} \cap \mathrm{lev}_{\le 1}(\|\cdot\|_G)$ by
definition of $\mathcal{S}$ in \eqref{eq.def.regionsinr}
and \eqref{eq.extention.rhodiagG} in Fact~\ref{fact.norm},
this claim follows from the claim (i) and $$\signal{s}^{\star} \in [0, b]^{N} \cap \mathrm{lev}_{\le 1}(\|\cdot\|_G) \subset \mathcal{S}.$$

(iv)
If $b$ is large enough,
the solutions of Problems \eqref{eq.sumrate.sinr} and \eqref{eq.refsumrate.frsinr}
are the same by Proposition \ref{prop.existence.equivalence.prob.sinr.box}. Thus,
this claim follows from the claim in (i)
by continuity of $\Phi$ and Lemma \ref{lemma.LipGradCost.sinr}(i).
\end{proof}

\subsection{Efficient subgradient computation method}
In this subsection, by focusing on the mappings $G$ considered in Example \ref{ex.maxlinearG},
we present a low-complexity method
for finding a subgradient of $\|\cdot\|_G$
used in Algorithm \eqref{eq:HSDMmaxsumrate.sinr} for Problem \eqref{eq.refsumrate.frsinr},
where we additionally assume the following condition, which is stronger than Assumption \ref{asmp.convex.rhoMl.exp} (see Remark \ref{remark.optimality.rate}).

\begin{assumption}
\label{asmp.convex.rhoMl}
	In addition to the setting of Example \ref{ex.maxlinearG},
we assume that
$\varrho_{\signal{M}_l}\colon\real_{+}^{N}\to\real_{+}$
is convex on $\real_{++}^{N}$ for every $l \in \{1,\ldots,L\}$.
\end{assumption}

\begin{remark}
\label{remark.conv.sinr.maxlinG}
If Assumption \ref{asmp.convex.rhoMl} holds, then
$G$ given in Example \ref{ex.maxlinearG} satisfies Assumption \ref{asmp.norm_inducing} by Fact \ref{fact.maxrhoconvex}.
\end{remark}

Since $\signal{s}_{k} \in \real_{++}^{N}$ holds at every iteration $k \in \Natural$
by Theorem \ref{theo.converge.sinr}(ii), it is sufficient to find a subgradient of $\|\cdot\|_G$
at a point in $\real_{++}^{N}$. To this end, we present the following proposition.

\begin{algorithm}[t]
\SetCommentSty{textrm}
%\DontPrintSemicolon
\setstretch{1.2}
\small
%\LinesNumbered
\caption{for Problem \eqref{eq.sumrate.sinr} using Example \ref{ex.maxlinearG}}
\label{alg.sinr.Gmaxlin}
\KwIn{$\signal{M}_l \in\real_{++}^{N\times N}~(l=1,\ldots,L)$, $\bm{w} \in \real_{++}^N$, $\signal{s}_1 \in \mathbb{R}_{++}^{N}$, $(\mu_k)_{k\in\mathbb{N}} \subset \real_{++}^{N}$
satisfying $\lim_{k\to\infty}\mu_k = 0$ and $\sum_{k\in\mathbb{N}}\mu_k = \infty$, and large enough $b >0$.}
\For{$k = 1,2,\ldots$}{
	$\gamma_k = \max_{l \in \{1,\ldots,L\}}\rho(\mathrm{diag}(\signal{s}_k)\signal{M}_l)$\;
	Choose $l_{k}^{\star} \in \argmax_{l \in \{1,\ldots,L\}}\rho(\mathrm{diag}(\signal{s}_k)\signal{M}_l)$\;
	\If{$\gamma_k > 1$}{
		Find $\signal{\xi}_k \in \real_{++}^{N}$ s.t.
		$\mathrm{diag}(\signal{s}_k)\signal{M}_{l_{k}^{\star}}\signal{\xi}_k = 
		\gamma_k\signal{\xi}_k$\;
		Find $\signal{\eta}_k \in \real_{++}^{N}$ s.t.
		$\signal{\eta}_k^{\mathsf{T}}\mathrm{diag}(\signal{s}_k)\signal{M}_{l_{k}^{\star}} = 
		\gamma_k\signal{\eta}_k^{\mathsf{T}}$\;
		$\signal{g}_k = (\mathrm{diag}(\signal{\eta}_k)\signal{M}_{l_{k}^{\star}}\signal{\xi})/(\signal{\eta}_k^{\mathsf{T}}\signal{\xi}_k)$\;
		$\hat{\signal{s}}_{k} = \signal{s}_k - (\gamma_k-1)\signal{g}_k/\|\signal{g}_k\|_2^2$\;
	}\Else{$\hat{\signal{s}}_{k} = \signal{s}_k$\;}
	$\tilde{\signal{s}}_{k} = P_{[0,b]^{N}}(\hat{\signal{s}}_{k})$\tcp*{See \eqref{eq.def.projectionbox}}
	$\signal{s}_{k+1} = \tilde{\signal{s}}_{k} - \mu_{k}\nabla\Phi(\tilde{\signal{s}}_{k})$\tcp*{See \eqref{eq.grad.sumrate.sinr}}
}
\end{algorithm}

\begin{proposition}
\label{prop.subgrad.sinr.maxlinG}
Consider the setting of Example \ref{ex.maxlinearG}.
Fix $\signal{s} \in \real_{++}^{N}$ arbitrarily, and
let $$l^{\star} \in
\argmax_{l \in \{1,\ldots,L\}} \varrho_{\signal{M}_l}(\signal{s})
= \argmax_{l \in \{1,\ldots,L\}}\rho(\mathrm{diag}(\signal{s})  \signal{M}_l).$$
Then, under Assumption \ref{asmp.convex.rhoMl}, we have
\begin{align}
\label{eq:ex.subgrad.sinr}
\frac{\mathrm{diag}(\signal{\eta})\signal{M}_{l^{\star}}\bm{\xi}}{\bm{\eta}^{\mathsf{T}} \bm{\xi}} \in \partial \|\cdot\|_G (\signal{s}),
\end{align}
where $\signal{\xi} \in \real_{++}^{N}$ and $\signal{\zeta} \in \real_{++}^{N}$
are right and left eigenvectors of $\mathrm{diag}(\signal{s})\signal{M}_{l^{\star}} \in \real_{++}^{N \times N}$
corresponding to the spectral radius $\rho(\mathrm{diag}(\signal{s})\signal{M}_{l^{\star}}) > 0$.
\end{proposition}
\begin{proof}
Owing to the equality shown in \eqref{eq.max.Gnorm},
$\signal{g} \in \partial \|\cdot\|_{G_{l^\star}}(\signal{s})$ implies
$\signal{g} \in \partial \|\cdot\|_{G}(\signal{s})$ by Fact \ref{fact.subdifmaxconv} in Appendix \ref{appendix.proof.subgradient}.
Hence it is sufficient to show that the left-hand side of \eqref{eq:ex.subgrad.sinr}
is a subgradient of $\|\cdot\|_{G_{l^\star}}$ at $\signal{s}$.
Since $\signal{M}_{l^\star} \in \real_{++}^{N \times N}$ holds in the setting of Example \ref{ex.maxlinearG},
by Fact \ref{fact.diff.rhodiagposM},
$\varrho_{\signal{M}_{l^\star}}$ is differentiable on $\real_{++}^{N}$,
and its gradient at $\signal{s} \in \real_{++}^{N}$ is given by the left-hand side of \eqref{eq:ex.subgrad.sinr}.
Since the equality in \eqref{eq.normG.specradlinar} holds in
a neighborhood of $\signal{s} \in \real_{++}^{N}$,
the left-hand side of \eqref{eq:ex.subgrad.sinr} is also the gradient of $\|\cdot\|_{G_{l^\star}}$ at $\signal{s}$, and hence is a (unique) subgradient of $\|\cdot\|_{G_{l^\star}}$ 
at $\signal{s}$.
\end{proof}

Algorithm \ref{alg.sinr.Gmaxlin} is a particular instance of \eqref{eq:HSDMmaxsumrate.sinr} using the subgradient characterized in Proposition \ref{prop.subgrad.sinr.maxlinG} under Assumption \ref{asmp.convex.rhoMl}.
Since Assumption \ref{asmp.convex.rhoMl} implies Assumption \ref{asmp.norm_inducing},
Theorem \ref{theo.converge.sinr} applies to Algorithm \ref{alg.sinr.Gmaxlin}
under Assumption \ref{asmp.convex.rhoMl}.
\end{document}